\documentclass[11pt]{article}

\usepackage{amsthm}
\usepackage{amsmath}
\usepackage{amssymb}
\usepackage{url}

\newcommand{\Ahat}{\hat{A}}
\newcommand{\Bhat}{\hat{B}}

\newcommand{\bw}{{\textstyle \bigwedge}}
\renewcommand{\d}{\delta}
\newcommand{\ce}{\mathcal{E}}
\newcommand{\cet}{\widetilde{\ce}}
\newcommand{\Cth}{Carath\'{e}odory}
\newcommand{\D}{\Delta}
\newcommand{\dt}{\tilde{d}}
\newcommand{\dT}{d_\mathrm{T}}
\newcommand{\eval}[2]{\left. #1 \right|_{#2}}
\newcommand{\Fb}{\fz{}{2}}
\newcommand{\Fm}{\fz{}{m}}
\newcommand{\Fmb}{\fz{2}{m}}
\newcommand{\Fmk}{\fz{k}{m}}

\newcommand{\fz}[2]{\mbox{$\overset{\hspace{0.1em}\scriptscriptstyle\o}{\raisebox{0ex}[1.2ex]{\phantom{T}}}\hspace{-0.7em}T^{#1}_{#2}$}}
\newcommand{\g}{\gamma}
\newcommand{\ghat}{\hat{\g}}
\newcommand{\GL}{\mathrm{GL}}
\newcommand{\gl}{\mathfrak{gl}}
\newcommand{\gt}{\tilde{\g}}

\newcommand{\id}{\mathrm{id}}

\newcommand{\iT}{i_{\mathrm{T}}}
\newcommand{\J}{\mathcal{J}}
\newcommand{\jbar}{\bar{\jmath}}

\renewcommand{\L}{\Lambda}
\newcommand{\Lie}{\mathcal{L}}
\renewcommand{\o}{\mathrm{o}}
\newcommand{\p}{\partial}
\newcommand{\pd}[2]{\frac{\p #1}{\p #2}}
\newcommand{\pdb}[3]{\frac{\p #1}{\p #2 \, \p #3}}

\newcommand{\pde}[3]{\eval{\pd{#1}{#2}}{#3}}
\newcommand{\pdeb}[4]{\eval{\pdb{#1}{#2}{#3}}{#4}}
\newcommand{\pdx}[4]{\frac{\p^{#1}{#2}}{\p #3 \cdots \p #4}}
\newcommand{\pdxa}[5]{\frac{\p^{#1}{#2}}{\p #3 \cdots \p #4 \, \p #5}}
\newcommand{\phibar}{\bar{\phi}}
\newcommand{\phihat}{\hat{\phi}}
\newcommand{\psihat}{\hat{\psi}}

\newcommand{\R}{\mathbb{R}}

\newcommand{\Sg}{\mathfrak{S}}

\newcommand{\tr}{\mbox{\textsc{t}}}
\newcommand{\Tht}{\widetilde{\Theta}}

\newcommand{\udot}{\dot{u}}
\newcommand{\varphit}{\tilde{\varphi}}
\newcommand{\ve}{\varepsilon}
\newcommand{\vf}[1]{\frac{\partial}{\partial #1}}
\newcommand{\vfe}[2]{\eval{\vf{#1}}{#2}}
\newcommand{\vt}{\vartheta}

\def\hook{\kern 3pt \vrule height 0pt depth 0.4pt width 3pt
\vrule height 5pt depth 0.4pt\kern 3pt}

\newtheorem{theorem}{Theorem}
\newtheorem{prop}[theorem]{Proposition}
\newtheorem{lem}[theorem]{Lemma}
\newtheorem{cor}[theorem]{Corollary}

\allowdisplaybreaks[1]

\newcommand{\art}[6]{#1: #2 {\it #3\/} {\bf #4} (#5) #6}

\newcommand{\book}[4]{#1: {\it #2\/} (#3, #4)}

\title{Homogeneous variational problems: a minicourse}

\author
{D.\,J.\ Saunders}

\begin{document}

\maketitle

\abstract{A Finsler geometry may be understood as a homogeneous variational problem, where the Finsler function is the Lagrangian. The extremals in Finsler geometry are curves, but in more general variational problems we might consider extremal submanifolds of dimension $m$. In this minicourse we discuss these problems from a geometric point of view.\\[2ex]
\textbf{MSC:} 35A15, 58A10, 58A20\\
\textbf{Keywords:} calculus of variations, parametric problems
}


\section{Introduction}

This paper is a written-up version of the major part of a minicourse given at the sixth Bilateral Workshop on Differential Geometry and its Applications, held in Ostrava in May~2011. Much of the discussion at these workshops is on Finsler geometry, where the interest is in variational problems defined on tangent manifolds by a `Finsler function', a smooth function defined on the slit tangent manifold (excluding the zero section) and satisfying certain homogeneity and nondegeneracy properties. The extremals of such problems are geometric curves in the original (base) manifold, without any particular parametrization but with an orientation.

For this particular workshop it was felt that it might be worthwhile to describe slightly more general problems, looking at variational problems where the extremals were submanifolds of dimension $m$, but where the action function still depended upon no more than the first derivatives of the submanifold~\cite{GH96},\cite{Run73}; for example, minimal surface problems would be included in this description. This minicourse introduces a version of the geometric background needed to express such problems, in terms of velocity manifolds. There is an alternative approach to such problems involving manifolds of contact elements (quotients of velocity manifolds); we refer to this only briefly, when we consider the action of the jet group.

Although we consider only first order variational problems, we nevertheless need to use second order velocities:\ for instance, the Euler-Lagrange equations for first order variational problems are second-order differential equations. We do this in a slightly unusual way, looking at a particular submanifold of the double velocity manifold. Having done this, we look at some geometrical and cohomological constructions, before obtaining a version of the first variation formula for variational problems with fixed boundary conditions. The final part of the minicourse, which considered various concepts of regularity, has been omitted from this paper for reasons of space; the concepts described may be found in a recent paper~\cite{CS09}. We give only a few other references:\ \cite{KMS93} provides extensive background material on various types of jet manifold and the actions of the jet groups; \cite{Sau09} introduces in a more general context the type of cohomological approach we use these types of variational problem; and~\cite{Sau10}, with a philosophy similar to that of the present paper, compares these problems with those defined on jets of sections of fibrations. 

I should like to thank the organisers of the Workshop for inviting me to give this course. I acknowledge the  support of grant no.\  201/09/0981 for Global Analysis and its Applications from the Czech Science Foundation; grant no.\ MEB 041005 for Finsler structures and the Calculus of Variations ; and also the joint IRSES project GEOMECH (EU FP7, nr 246981).


\section{Velocities}

In this section we see how to construct manifolds of first order and second order velocities, and also how certain groups, the jet groups, act on these manifolds.


\subsection{First order velocities}

Let $E$ be a connected, paracompact, Hausdorff manifold of class $C^\infty$ and of finite dimension $n$; let $O \subset \R^m$ (with $m < n$) be open and connected, with $0 \in O$. A map $\g : O \to E$ will be called an \emph{$m$-curve in $E$}. The $1$-jet $j^1_0\g$ of $\g$ at zero will be called a \emph{velocity} (or \emph{$m$-velocity}), and the set $T_m E = \{j^1_0\g\}$ of velocities of all $m$-curves in $E$ will be called the \emph{velocity (or $m$-velocity) manifold of $E$}. We map $T_m E$ to $E$ by
\[
\tau_{mE} : T_m E \to E \, , \qquad \tau_{mE}(j^1_0\g) = \g(0) \, .
\]

We shall show that $T_m E$ really is a manifold (and is connected, paracompact and Hausdorff, and indeed is a vector bundle over $E$) by identifying it with the Whitney sum over $E$ of $m$ copies of the tangent manifold $TE$.
\begin{lem}
\label{L1}
There is a canonical identification $T_m E \cong \bigoplus^m TE$.
\end{lem}
\begin{proof}
Let $i_k : \R \to \R^m$ be the inclusion $i_k(s) = (0, \ldots, 0, s, 0, \ldots, 0)$. Then each $\g \circ i_k$ is a curve in $E$, and the map
\[
j^1_0\g \mapsto \bigl( j^1_0(\g \circ i_1), \ldots, j^1_0(\g \circ i_m) \bigr)
\]
is a bijection $T_m E \to \bigoplus^m TE$ preserving the fibration over $E$.
\end{proof}
\begin{cor}
\label{C2}
Let $\{dt^i\}$ be the canonical basis of $\R^{m*}$; then
\[
T_m E \to TE \otimes \R^{m*} \, , \qquad (\xi_1, \ldots, \xi_m) \mapsto \xi_i \otimes dt^i
\] 
is a vector bundle isomorphism. \qed
\end{cor}
If $(U; u^a)$ is a chart on $E$ then $(U^1; u^a, u^a_i)$ is a chart on $T_m E$, where
\[
U^1 = \tau_{mE}^{-1}(U) \, , \qquad u^a_i(j^1_0\g) = D_i \g^a (0) = D_i (u^a \circ \g) (0) \, . 
\]
If $j^1_0\g = (\xi_1, \ldots, \xi_m)$ then it is clear that $u^a_i(j^1_0\g) = \udot^a(\xi_i)$. The rule for changing coordinates on $T_m E$ is therefore
\[
v^b_i(j^1_0\g) = \pde{v^b}{u^a}{\g(0)} u^a_i(j^1_0\g) \, .
\]
We can see from this that the superscript $a$ labeling the coordinate function $u^a_i$ depends on the original choice of chart $u^a$ on $E$, whereas the subscript $i$ is independent of this choice and so is the index of a component of the velocity (namely, the tangent vector $\xi_i$). We call indices of this latter type \emph{counting indices} rather than \emph{coordinate indices}.

We shall be particularly interested in the subsets of $T_m E$ containing those velocities $j^1_0\g$ where the $m$-curve $\g$ has certain properties. Write $\Fm E$ for the subset
\[
\{ j^1_0\g \in T_m E : \g \text{ is an immersion near zero} \} \, ;
\]
if $j^1_0\g = (\xi_1, \ldots, \xi_m)$ and $j^1_0\g \in \Fm E$ then $\{ \xi_1, \ldots, \xi_m \}$ will be linearly independent. An element of $\Fm E \subset T_m  E$ will be called a \emph{regular velocity}.
\begin{prop}
The regular velocities form an open-dense submanifold.
\end{prop}
\begin{proof}
To show that $\Fm E$ is open in $T_m E$, define the map $\wedge : T_m E \to \bigwedge^m TE$ by $(\xi_1, \ldots, \xi_m) \mapsto \xi_1 \wedge \cdots \wedge \xi_m$. Then
\begin{itemize}
\item The map $\wedge$ is fibred over the identity on $E$ and is continuous (it is polynomial in the fibre coordinates $u^a_i$);
\item $j^1_0\g \in \Fm E$ exactly when $\wedge(j^1_0\g) \ne 0$;
\item the zero section of $\bigwedge^m TE$ is closed.
\end{itemize}
To show that $\Fm E$ is dense in $T_m E$, define the map $f : U^1 \to \R$ by $f(j^1_0\g) = \det \bigl( u^j_i(j^1_0\g) \bigr)$, where $(u^j_i)$ is the $m \times m$ submatrix containing the first $m$ rows of the $n \times m$ matrix $u^a_i$. If $j^1_0\g \in O \subset U^1$ where $O$ is open and $O \cap \Fm E = \emptyset$ then $f$ vanishes on $O$. But
\[
\eval{\frac{\p^m f}{\p u^1_1 \, \p u^2_2 \cdots \p u^m_m}}{j^1_0\g} = 1 \, .
\]
\end{proof}


\subsection{Second order velocities}

We define a \emph{second-order $m$-velocity} in the same way as a 2-jet at zero of an $m$-curve, and write
\[
T^2_m E = \{j^2_0\g \} \, , \qquad
\Fmb E = \{j^2_0\g : \g \text{ is an immersion near zero} \} \, .
\]
We also let $\tau^2_{mE} : T^2_m E \to E$,  $\tau^{2,1}_{mE} : T^2_m E \to T_m E$ be the projections
\[
\tau^2_{mE}(j^2_0\g) = \g(0) \, , \qquad \tau^{2,1}_{mE}(j^2_0\g) = j^1_0\g \, .
\]
We take charts on $T^2_m E$ to be $(U^2; u^a, u^a_i, u^a_{ij})$ where $U^2 = (\tau^2_{mE})^{-1}(U)$ and
\[
u^a_i(j^2_0\g) = D_i \g^a(0) \, , \qquad u^a_{ij}(j^2_0\g) = D_i D_j \g^a(0)
\]
so that $u^a_{ij} = u^a_{ji}$ (this constraint will cause complications in certain coordinate formul\ae). These charts form an atlas such that $T^2_m E$ becomes a manifold with the standard properties. We shall not demonstrate this directly; we shall show instead that it may be identified with a closed submanifold of a larger manifold, the manifold of double velocities.


\subsection{Double velocities}

We know that $T_m E$ is a manifold, so it has its own velocity manifold
\[
T_{m^\prime}T_m E = \{ j^1_0\gt \}
\]
where $\gt$ is an $m^\prime$-curve in $T_m E$. This is the \emph{$(m^\prime,m)$ double velocity manifold}. Charts on $T_{m^\prime}T_m E$ are therefore
\[
\bigl( (U^1)^1; u^a, u^a_i, u^a_{;j}, u^a_{i;j} \bigr) \, ,
\]
where $1 \le i \le m$ and $1 \le j \le m^\prime$, corresponding to the charts $(U^1; u^a, u^a_i)$ on $T_m E$. In most applications we have either $m^\prime = m$ or $m^\prime = 1$. We shall be interested in a particular submanifold of double velocities, known as holonomic double velocities.


\subsection{Holonomic double velocities}

If $\g$ is an $m$-curve in $E$ then its \emph{prolongation} is the $m$-curve $\jbar^1\g$ in $T_m E$ where
\[
\jbar^1\g(t) = j^1_0(\g \circ \tr_t)
\]
and $\tr_t : \R^m \to \R^m$ is the translation map $\tr_t(s) = t + s$. Thus $j^1_0 \jbar^1\g \in T_m T_m E$. We use the notation $\jbar^1\g$ rather than $j^1\g$; the latter would be a map satisfying $j^1\g(t) = j^1_t\g$ whose codomain would be a set containing jets at arbitrary points of $\R^m$ rather than just at zero.
\begin{prop}
\label{P4}
The map
\[
\iota : T^2_m E \to T_m T_m E \, , \qquad \iota(j^2_0\g) = j^1_0 \jbar^1\g
\]
is an injection. Its image is the submanifold described in coordinates by
\[
u^a_i = u^a_{;i} \, , \qquad u^a_{i;j} = u^a_{j;i} \, .
\]
The image of the chart $(U^2; u^a, u^a_i, u^a_{ij})$ under the injection is the restriction of the chart $\bigl( (U^1)^1; u^a, u^a_i, u^a_{;j}, u^a_{i;j} \bigr)$ to the submanifold.
\end{prop}
\begin{proof}
Suppose $\g_1$, $\g_2$ are two $m$-curves in $E$ such that $j^1_0 \jbar^1\g_1 = j^1_0 \jbar^1\g_2$. Then for $\g_1$
\begin{align*}
u^a(j^2_0\g_1) & = u^a(\g_1(0)) = u^a(\jbar^1\g_1(0)) = u^a(j^1_0 \jbar^1\g_1) \, ; \\
u^a_i(j^2_0\g_1) & = D_i(u^a \circ \g_1)(0) = u^a_i(\jbar^1\g_1(0)) = u^a_i(j^1_0 \jbar^1\g_1) \, ; \\
u^a_{ij}(j^2_0\g_1) & = D_i D_j(u^a \circ \g_1)(0) = D_i(u^a_j \circ \jbar^1\g_1)(0) = u^a_{j;i}(j^1_0 \jbar^1\g_1)
\end{align*}
and similarly for $\g_2$, so that $j^2_0\g_1 = j^2_0\g_2$ and the map is an injection.

For any $m$-curve $\g$ in $E$
\[
u^a_{;i}(j^1_0 \jbar^1\g) = D_i(u^a \circ \jbar^1\g)(0) = D_i(u^a \circ \g)(0) = u^a_i(\jbar^1\g(0)) = u^a_i(j^1_0 \jbar^1\g)
\]
and
\[
u^a_{j;i}(j^1_0 \jbar^1\g) = D_i(u^a_j \circ \jbar^1\g)(0) = D_i(D_j(u^a \circ \g))(0)
\]
so that $u^a_i = u^a_{;i}$ and $u^a_{i;j} = u^a_{j;i}$ when restricted to the image of the injection.

Furthermore, if $\gt$ is an $m$-curve in $T_m E$ satisfying
\[
u^a_i(j^1_0\gt) = u^a_{;i}(j^1_0\gt) \, , \qquad u^a_{i;j}(j^1_0\gt) = u^a_{j;i}(j^1_0\gt)
\]
then the $m$-curve $\g$ in $E$ given in coordinates near $\tau_{mE}(\gt(0))$ by
\[
\g^a(t) = u^a(j^1_0\gt) + u^a_i(j^1_0\gt) t^i + \tfrac{1}{2} u^a_{i;j}(j^1_0\gt) t^i t^j
\]
so that $j^1_0 \jbar^1\g = j^1_0\gt$; thus the image of the injection is described locally by the equations $u^a_i = u^a_{;i}$, $u^a_{i;j} = u^a_{j;i}$ and is therefore a submanifold of $T_m T_m E$.

The relationship between the charts $(U^2; u^a, u^a_i, u^a_{ij})$ and $\bigl( (U^1)^1; u^a, u^a_i, u^a_{;j}, u^a_{i;j} \bigr)$ is immediate.
\end{proof}
The image of $T^2_m E$ in $T_m T_m E$ is called the submanifold of \emph{holonomic} double velocities. There is no canonical projection $T_m T_m E \to T^2_m E$; we may, however, consider a tubular neighbourhood $\nu : N \to T^2_m E$ of $T^2_m E$ in $T_m T_m E$, and then the condition $\nu\circ\iota = \id_{T^2_m E}$ (where $\iota : T^2_m E \to T_m T_m E$ is the injection) gives rise to the constraints
\begin{align*}
\pd{\nu^a}{u^c} & = \d^a_c \, , &
\pd{\nu^a}{u^c_p} + \pd{\nu^a}{u^c_{;p}} & = 0 \, , &
\pd{\nu^a}{u^c_{p;q}} + \pd{\nu^a}{u^c_{q;p}} & = 0 \\[3ex]
\pd{\nu^a_i}{u^c} & = 0 \, , &
\pd{\nu^a_i}{u^c_p} + \pd{\nu^a_i}{u^c_{;p}} & = \d^a_c \d^p_i \, , &
\pd{\nu^a_i}{u^b_{p;q}} + \pd{\nu^a_i}{u^b_{q;p}} & = 0 \\[3ex]
\pd{\nu^a_{ij}}{u^c} & = 0 \, , &
\pd{\nu^a_{ij}}{u^c_p} + \pd{\nu^a_{ij}}{u^c_{;p}} & = 0 \, , &
\pd{\nu^a_{ij}}{u^c_{p;q}} + \pd{\nu^a_{ij}}{u^c_{q;p}} & = \d^a_c (\d^p_i \d^q_j + \d^p_j \d^q_i) \, .
\end{align*}
for the coordinates of $\nu$, and hence to the conditions
\begin{align*}
d\nu^a & = du^a + \pd{\nu^a}{u^c_p} (du^c_p - du^c_{;p}) + \tfrac{1}{2} \pd{\nu^a}{u^c_{p;q}} (du^c_{p;q} - du^c_{q;p}) \\
d\nu^a_i & = \tfrac{1}{2} (du^a_i + du^a_{;i})
+ \tfrac{1}{2} \biggl( \pd{\nu^a_i}{u^c_p} - \pd{\nu^a_i}{u^c_{;p}} \biggr) (du^c_p - du^c_{;p}) 
+ \tfrac{1}{2} \pd{\nu^a_i}{u^c_{p;q}} (du^c_{p;q} - du^c_{q;p}) \\
d\nu^a_{ij} 
& = \tfrac{1}{2} (du^a_{i;j} +  du^a_{j;i}) + \pd{\nu^a_{ij}}{u^c_p} (du^c_p - du^c_{;p})
+ \tfrac{1}{2} \pd{\nu^a_{ij}}{u^c_{p;q}} (du^c_{p;q} - du^c_{q;p}) \, .
\end{align*}
We shall use these conditions later on.


\subsection{The exchange map}

There is another way of describing the submanifold of holonomic velocities.

A map $\psi : O^\prime \times O \to E$, where $O \subset \R^m$, $O^\prime \subset \R^{m^\prime}$  are open and connected, and where $0_{\R^m} \in O$ and $0_{\R^{m^\prime}} \in O^\prime$, is called a \emph{double $(m^\prime,m)$-curve}. For each $s\in O^\prime$
\[
\psi_s : O \to E \, , \qquad \psi_s(t) = \psi(s,t)
\]
is then an $m$-curve in $E$, so that $j^1_0\psi_s \in T_m E$. Thus
\[
j^1_0 (s \mapsto j^1_0\psi_s) \in T_{m^\prime} T_m E \, .
\]
\begin{lem}
\label{L5}
The \emph{exchange map} $e : T_{m^\prime} T_m E \to T_m T_{m^\prime}E$ is well-defined by $\psi \mapsto \psihat$ where $\psihat(t,s) = \psi(s,t)$ and is a smooth bijection.
\end{lem}
\begin{proof}
The element of $T_m T_m E$ defined by $\psi$ satisfies
\begin{align*}
u^a(j^1_0 (s \mapsto j^1_0\psi_s)) & = u^a(j^1_0\psi_0) = \psi^a_0(0) = \psi^a(0,0) \, , \\
u^a_i(j^1_0 (s \mapsto j^1_0\psi_s)) & = u^a_i(j^1_0\psi_0) = D_i(u^a \circ \psi_0)(0) = D_{2:i} \psi^a(0,0) \, , \\
u^a_{;j}(j^1_0 (s \mapsto j^1_0\psi_s)) & = D_j(u^a \circ (s \mapsto j^1_0\psi_s))(0) = D_j(s \mapsto \psi^a_s)(0) 
= D_{1:j} \psi^a(0,0) \, , \\
u^a_{i;j}(j^1_0 (s \mapsto j^1_0\psi_s)) & = D_j(u^a_i \circ (s \mapsto j^1_0\psi_s))(0) \\
& \qquad  = D_j(s \mapsto D_i \psi^a_s(0))(0) \, ,
= D_{1:j}D_{2:i} \psi^a(0,0)
\end{align*}
and carrying out the same calculation for $\psihat$ shows that $e$ is a well-defined injection. It is clearly an involution, and hence is a bijection. The coordinate formul\ae
\[
u^a \circ e = u^a \, , \qquad u^a_i \circ e = u^a_{;i} \, , \qquad u^a_{;j} \circ e = u^a_j \, , \qquad u^a_{i;j} \circ e = u^a_{j;i}
\]
show that it is smooth.
\end{proof}
\begin{prop}
The holonomic submanifold of $T_m T_m E$ is the fixed point set of the exchange map.
\end{prop}
\begin{proof}
This is immediate from the coordinate formul\ae\ for $e$.
\end{proof}


\subsection{Jet groups}

If we consider $m$-curves in $\R^m$ rather than in some other manifold, then we have the possibility of composing two such $m$-curves. If we insist that the origin must map to itself then the composition will always exist, although possibly with a smaller domain then the domains of the two original $m$-curves. We shall want the jets of these $m$-curves to have inverses, so that the curves themselves will need to be immersions near zero; it is convenient to assume that they are, in fact, diffeomorphisms onto their images.

So let $O \subset \R^m$ be open and connected with $0 \in O$, and let $\phi : O \to \phi(O) \subset \R^m$ be a diffeomorphism with $\phi(0) = 0$. The \emph{first and second order jet groups} are
\[
L^1_m = \{ j^1_0\phi \} \, , \qquad
L^2_m = \{ j^2_0\phi \} \, .
\]
The products for $L^1_m$ and $L^2_m$ are given by
\[
j^1_0\phi_1 \cdot j^1_0\phi_2 = j^1_0(\phi_1 \circ \phi_2) \, , \qquad
j^2_0\phi_1 \cdot j^2_0\phi_2 = j^2_0(\phi_1 \circ \phi_2) \, .
\]
\begin{lem}
The product rules define group structures on $L^1_m$ and $L^2_m$.
\end{lem}
\begin{proof}
The products are well-defined because the first (or second) derivatives of a composite depend only upon the first (or second) derivatives of the individual maps, by the first (or second) order chain rule; sssociativity of the products is inherited from that of composition. The diffeomorphism $\id_{\R^m}$ satisfies
\[
j^1_0(\id_{\R^m}) \cdot j^1_0\phi = j^1_0(\id_{\R^m} \circ \phi) = j^1_0\phi \, ;
\]
the map $\phibar : \phi(O) \to O$ given by $\phibar = \phi^{-1}$ satisfies $\phibar(0) = 0$, and 
\[
j^1_0\phibar \cdot j^1_0\phi = j^1_0(\phibar \circ \phi) = j^1_0(\id_O) = j^1_0(\id_{\R^m}) \, .
\]
Similar formul\ae\ hold for second-order jets.
\end{proof}
The map $L^1_m \to \R^{m^2}$, $j^1_0\phi \mapsto \bigl( D_j \phi^i(0) \bigr)$ defines global coordinates on $L^1_m$, and identifies it with $\GL(m,\R)$. The map $L^2_m \to \R^{m^2(m+3)/2}$, $j^2_0\phi \mapsto \bigl( D_j \phi^i(0), D_j D_k \phi^i(0) \bigr)$ defines global coordinates on $L^2_m$. Writing
\[
A^i_j = D_j \phi^i(0) \, , \qquad B^i_{jk} = D_j D_k \phi_i(0)
\]
where $\det A^i_j \ne 0$ because $\phi$ is a diffeomorphism, the product rule in $L^1_m$ is
\[
(A\Ahat)^i_j = A^i_h \Ahat^h_j
\]
and the product rule in $L^2_m$ is
\begin{align*}
\bigl( (A,B)(\Ahat,\Bhat) \bigr)^i_j & = A^i_h \Ahat^h_j \, , \\
\bigl( (A,B)(\Ahat,\Bhat) \bigr)^i_{jk} & = A^i_l \Bhat^l_{jk} + B^i_{hl} \Ahat^h_j \Ahat^l_k \, ,
\end{align*}
the latter formula arising from the second order chain rule
\begin{align*}
D_j D_k (\phi\phihat)^i(0) & = D_j(D_l\phi^i \circ \phihat) D_k\phihat^l)(0) \\ & 
= D_l\phi^i(0) D_j D_k \phihat^l(0) +  D_h D_l \phi^i(0) D_j\phihat^h(0) D_k\phihat^l(0) 
\end{align*}
using $\phi(0) = \phihat(0) = 0$.
\begin{cor}
The groups $L^1_m$ and $L^2_m$ are Lie groups. \qed
\end{cor}
\begin{lem}
The \emph{oriented} subgroups $L^{1+}_m$ and $L^{2+}_m$, where $\phi$ preserves orientation, are connected. 
\end{lem}
\begin{proof}
As $L^1_m$ may be identified with $\GL(m,\R)$, the subgroup $L^1_m$ where $\phi$ preserves orientation may be identified with $\GL^+(m,\R)$, the subgroup of matrices satisfying $\det A^i_j > 0$, which is connected.

The map $L^1_m \to L^2_m$ given by $j^1_0\phi \mapsto j^2_0\phihat$, where $\phihat$ is the linear map $\phihat^i(t) = A^i_j t^j$ with $(A^i_j)$ being the matrix corresponding to $j^1_0\phi$, is continuous; the coordinates of the image are $(A^i_j, 0)$. The image of the subgroup $L^{1+}_m$ under this map is therefore connected. But every element of $L^{2+}_m$ may be joined to an element of this image by a path given in coordinates by
\[
s \mapsto (A^i_j, sB^i_{jk}) \, , \qquad s \in [0,1]
\] 
\end{proof}


\subsection{Group actions}

The jet groups $L^1_m$ and $L^2_m$ act on the velocity manifolds $T_m E$ and $T^2_m E$ by
\[
(j^1_0\phi, j^1_0\g) \mapsto j^1_0(\g\circ\phi) \, , \qquad
(j^2_0\phi, j^2_0\g) \mapsto j^2_0(\g\circ\phi) \, .
\]
These are right actions, and in coordinates they are
\begin{align*}
u^a & \mapsto u^a \\
u^a_i & \mapsto u^a_h A^h_i \\
u^a_{ij} & \mapsto u^a_{hk} A^h_i A^k_j + u^a_h B^h_{jk}
\end{align*}
where $A^i_j$ and $B^i_{jk}$ are the global coordinates of $j^2_0\phi$.
\begin{lem}
The action of $L^1_m$ on $T_m E$ restricts to $\Fm E$, and the restricted action is free. The action of $L^2_m$ on $T^2_m E$ restricts to $\Fmb E$, and the restricted action is free.
\end{lem}
\begin{proof}
The map $\phi$ is a diffeomorphism onto its image, so if $\g$ is an immersion near zero then so is $\g\circ\phi$.

We use coordinates to show that the restricted actions are free. Suppose first that $j^1_0(\g\circ\phi) = j^1_0\g$, so that
\[
u^a_j(j^1_0\g) = u^a_i(j^1_0\g) A^i_j \, ;
\]
as $\g$ is an immersion near zero and $u^a_i(j^1_0\g) = D_i \g^a(0)$, it follows that the $m \times n$ matrix $u^a_i(j^1_0\g)$ must have rank $m$, so that $A^i_j = \d^i_j$ and hence $j^1_0\phi = j^1_0(\id_{\R^m})$.

Now suppose that $j^2_0(\g\circ\phi) = j^2_0\g$, so that $u^a_j(j^2_0\g) = u^a_i(j^2_0\g) A^i_j$ and now also
\[
u^a_{hk}(j^2_0\g) = u^a_{ij}(j^2_0\g) A^i_h A^j_k + u^a_i(j^2_0\g) B^i_{hk} \, .
\]
As before we see that $A^i_j = \d^i_j$, so that
\[
u^a_{hk}(j^2_0\g) = u^a_{hk}(j^2_0\g) + u^a_i(j^2_0\g) B^i_{hk}
\]
and therefore that $u^a_i(j^2_0\g) B^i_{hk} = 0$; the rank condition on $u^a_i(j^2_0\g)$ now tells us that $B^i_{hk} = 0$.
\end{proof}


\subsection{Infinitesimal actions}

Let $(a^i_j)$ be an element of the Lie algebra of $L^1_m$; the identification of the group with $\GL(m,r)$ means that its Lie algebra may be identified with $\gl(m,\R)$ so that $(a^i_j)$ is an arbitrary $m \times m$ matrix.
\begin{lem}
The vector field on $T_m E$ corresponding to $(a^i_j)$ is
\[
a^i_j u^a_i \vf{u^a_j} \, .
\]
\end{lem}
\begin{proof}
The map $\sigma : (-\ve,\ve) \to \GL(m,\R)$, defined for sufficiently small $\ve$ by $\sigma(s) = (\d^i_j + s a^i_j)$, is a curve in $\GL(m,\R)$ whose tangent vector at the identity is $(a^i_j)$. If $j^1_0\g \in T_m E$ then the corresponding curve through $j^1_0\g$ is given in coordinates by
\[
s \mapsto \left( u^b(j^1_0\g), (\d^i_j + s a^i_j) u^b_i(j^1_0\g) \right) \, .
\]
The resulting tangent vector $\xi \in T_{j^1_0\g} T_m E$ satisfies
\[
\udot^b(\xi) = 0 \, , \qquad \udot^b_j(\xi) = a^i_j u^b_i(j^1_0\g)
\]
so that the vector field on $T_m E$ defined by the Lie algebra element $(a^i_j)$ is
\[
a^i_j u^b_i \vf{u^b_j} \, .
\]
\end{proof}
We write $d^j_i$ for the Lie derivative operation of the basis vector field $\D^j_i = u^a_i \p / \p u^a_j$.


\subsection{Second order infinitesimal actions}

There is a similar result for the action of the Lie algebra of $L^2_m$.
\begin{lem}
Let $(a^i_j, b^i_{jk})$ be an element of the Lie algebra of $L^2_m$. The corresponding vector field on $T^2_m E$ is
\[
a^i_j u^a_i \vf{u^a_j} + \frac{1}{\#(jk)} \bigl( 2a^i_j u^a_{ik} + b^i_{jk} u^a_i \bigr) \vf{u^a_{jk}} \, .
\]
where $\#(jk)$ equals 1 if $j=k$ and equals 2 otherwise.
\end{lem}
\begin{proof}
Let $\g$ be the curve in $L^2_m$ through the identity $j^2_0(\id)$ given in coordinates by
\[
s \mapsto \bigl( \d^i_j + s a^i_j, s b^i_{jk} \bigr) \, .
\]
If $j^2_0\g \in T^2_m E$ then the corresponding curve through $j^2_0\g$ is given in coordinates by
\[
s \mapsto \left( u^a(j^2_0\g), u^a_i(j^2_0\g) (\d^i_j + s a^i_j), 
u^a_{hi}(j^2_0\g) (\d^h_j + s a^h_j) (\d^i_k + s a^i_k) + s u^a_i(j^2_0\g) b^i_{jk} \right) \, .
\]
The resulting tangent vector $\xi \in T_{j^2_0\g} T^2_m E$ satisfies
\begin{align*}
\udot^a(\xi) & = 0 \\
\udot^a_j(\xi) & = a^i_j u^a_i(j^2_0\g) \\
\udot^a_{jk}(\xi) & = a^i_k u^a_{ij}(j^2_0\g) + a^i_j u^a_{ik}(j^2_0\g) + b^i_{jk} u^a_i(j^2_0\g)
\end{align*}
so that the vector field on $T^2_m E$ defined by the Lie algebra element corresponding to $(a^i_j,b^i_{jk})$ is
\[
a^i_j u^a_i \vf{u^a_j} + \frac{1}{\#(jk)} \bigl( 2 a^i_j u^a_{ik} + u^i_{jk} u^a_i \bigr) \vf{u^a_{jk}} \, .
\]
\end{proof}
We write $d^j_i$ and $d^{jk}_i$ for the Lie derivative operation of the basis vector fields
\[
\D^j_i = u^a_i \vf{u^a_j} + \frac{2}{\#(jk)} u^a_{ik} \vf{u^a_{jk}} \, , \qquad
\D^{jk}_i = \frac{1}{\#(jk)} u^a_i \vf{u^a_{jk}} \, .
\]
Note the use of the symbol $\#(jk)$ to compensate for the fact that the coordinate functions $u^a_{jk}$ and $£u^a_{kj}$ are equal, so that summing over $j$ and $k$ could result in double-counting.


\section{Geometric structures}

The special structure of velocity manifolds manifests itself in the existence of certain differential operators (`total derivatives') and differential forms (`contact forms') which capture certain aspects of the structure. The total derivatives and contact forms may also be used to identify those maps between velocity manifolds, and vector fields on velocity manifolds, which have been constructed by a process known as prolongation. Finally, there is an algebraic method of lifting tangent vectors from a manifold to its velocity manifold called the vertical lift, and this gives rise to vertical endomorphisms.


\subsection{Total derivatives}

The identity map $T_m E \to T_m E$ defines a section of the pull-back bundle $\tau_{mE}^* T_m E \to T_m E$. Its components $d_i$ are the \emph{total derivatives}, vector fields along $\tau_{mE}$. At a point $j^1_0\g$, the identification $T_m E \cong \bigoplus^m TE$ from Lemma~\ref{L1} gives the $k$-th component of $j^1_0\g$ as
\[
\eval{d_k}{j^1_0\g} = j^1_0(\g\circ i_k) = T\g(j^1_0 i_k) = T\g \biggl( \vfe{t^k}{0} \biggr) \, .
\]
Note that the subscript $k$ is a counting index, not a coordinate index. In coordinates, if $f$ is a function on $E$ then
\begin{align*}
\eval{d_k f}{j^1_0\g} & = \eval{d_k}{j^1_0\g} f = T\g \biggl( \vfe{t^k}{0} \biggr) f = \pde{(f \circ \g)}{t^k}{0} \\
& \qquad \qquad = \pde{f}{u^a}{\g(0)} D_k\g^a(0) = u^a_k(j^1_0\g) \pde{f}{u^a}{\g(0)}
\end{align*}
so that
\[
d_k = u^a_k \vf{u^a} \, .
\]
It is clear from this coordinate formula that the image of $(d_1, \ldots, d_m)$, a subspace of $T_{\g(0)}E$ corresponding to each point $j^1_0\g \in T_m E$, does not have constant rank on $T_m E$. But its restriction to $\Fm E$, where the $m \times n$ matrix $u^a_i$ has maximal rank, \emph{does} have constant rank $m$.


\subsection{Second order total derivatives}

We take a similar approach to second order total derivatives. The inclusion map $T^2_m E \to T_m T_m E$ defines a section of the pull-back bundle $\tau^{2,1\,*}_{mE} T_m T_m E \to T^2_m E$; its components $d_i$ are the \emph{second order total derivatives}, vector fields along $\tau^{2,1}_{mE}$. At a point $j^2_0\g$,
\[
\eval{d_k}{j^2_0\g} = T(j^1\g) \biggl( \vfe{t^k}{0} \biggr) \, ;
\]
in coordinates
\[
d_k = u^a_k \vf{u^a} + u^a_{kj} \vf{u^a_j} \, .
\]
Once again the image of $(d_1, \ldots, d_m)$, a subspace of $T_{j^1_0\g}T^m E$ corresponding to each point $j^2_0\g \in T^2_m E$, does not have constant rank on $T^2_m E$, but its restriction to $\Fmb E$ \emph{does} have constant rank $m$.


\subsection{Contact 1-forms}

Contact 1-forms on $T_m E$ or on $T^2_m E$ are the horizontal 1-forms which annihilate total derivatives, so that $\theta$ is a contact $1$-form exactly when
\[
\langle \theta, d_k \rangle = 0 \, .
\]
Here, `horizontal' means horizontal over $E$ for a $1$-form on $T_m E$, and it means horizontal over $T_m E$ for a $1$-form on $T^2_m E$, so that it makes sense to evaluate such forms on total derivatives; indeed, the modules of such horizontal $1$-forms are dual to the modules of vector fields along $T_m E \to E$ or along $T^2_m E \to T_m E$.

In fact we shall consider contact $1$-forms, not on the whole of $T_m E$ or $T^2_m E$, but on the submanifolds of regular velocities $\Fm E$ and $\Fmb E$. The reason is that, as mentioned previously, the image of the map $(d_1, \ldots, d_m)$ has constant rank $m$ only on the regular submanifolds; it is, for example, zero on the zero section of $T_m E$, and so every horizontal cotangent vector on that zero section is annihilated by all the total derivatives. If we were to include non-regular velocities then there would be `contact' cotangent vectors which were not the values of any (smooth, and hence continuous) contact $1$-form. 

The important property of contact $1$-forms is that they always pull back to zero under prolongations.
\begin{lem}
If $\theta$ is a contact $1$-form on $\Fm E$ then $(\jbar^1\g)^* \theta = 0$. If it is a contact $1$-form on $\Fmb E$ then $(\jbar^2\g)^* \theta = 0$, where the prolonged $m$-curve $\jbar^2\g$ is defined by $\jbar^2\g(t) = j^2_0(\g\circ\tr_t)$.
\end{lem}
\begin{proof}
If $\theta$ is a contact $1$-form on $\Fm E$ then 
\begin{align*}
\biggl\langle \eval{(\jbar^1\g)^* \theta}{t} \, , \vfe{t^k}{t} \bigg\rangle 
& = \biggl\langle \eval{(j^1(\g\circ\tr_t))^* \theta}{t} \, , \vfe{t^k}{t} \bigg\rangle \\
& = \biggl\langle \eval{(j^1\g)^* \theta}{0} \, , \vfe{t^k}{0} \bigg\rangle \\
& = \biggl\langle \eval{\theta}{j^1_0\g} \, , T\g \biggl( \vfe{t^k}{0} \biggr) \bigg\rangle \\
& = \langle \eval{\theta}{j^1_0\g} \, , \eval{d_k}{j^1_0\g} \rangle = 0 \, .
\end{align*}
The proof for a contact $1$-form on $\Fmb E$ is similar.
\end{proof}
\begin{prop}
If $\theta$ is a $1$-form on $\Fm E$ satisfying $(\jbar^1\g)^* \theta = 0$ for every prolonged $m$-curve $\jbar^1\g$ in $\Fm E$ then $\theta$ is horizontal over $E$, and is a contact $1$-form. A similar result holds for contact $1$-forms on $\Fmb E$.
\end{prop}
\begin{proof}
We show first that $\theta$ is horizontal over $E$, by showing that it is horizontal at each point $j^1_0\g \in \Fm E$. Write $\theta$ in coordinates around such a point as
\[
\theta = \theta_a du^a + \theta^i_a du^a_i \, ;
\]
then if $\g$ is a representative $m$-curve for the velocity $j^1_0\g$ we have
\[
(\jbar^1\g)^* \theta = (\theta_a \circ \jbar^1\g) \bigl( (\jbar^1\g)^* du^a \bigr) 
+ (\theta^i_a \circ \jbar^1\g) \bigl( (\jbar^1\g)^* du^a_i \bigr) \, .
\]
But
\begin{align*}
\eval{(\jbar^1\g)^* du^a}{0} & = \eval{d(u^a \circ \jbar^1\g)}{0} = \eval{d\g^a}{0} = \pde{\g^a}{t^j}{0} \eval{dt^j}{0} \\
\eval{(\jbar^1\g)^* du^a_i}{0} & = \eval{d(u^a_i \circ \jbar^1\g)}{0} = \eval{d \biggl( \pd{\g^a}{t^i} \biggr)}{0}
= \pdeb{\g^a}{t^i}{t^j}{0} \eval{dt^j}{0}
\end{align*}
so that
\[
0 = \eval{(\jbar^1\g)^* \theta}{0} = \biggl( (\theta_a \circ \jbar^1\g)(0) \pde{\g^a}{t^j}{0}
+ (\theta^i_a \circ \jbar^1\g)(0) \pdeb{\g^a}{t^i}{t^j}{0} \biggr) \eval{dt^j}{0}
\]
and hence
\[
\theta_a(j^1_0\g) \pde{\g^a}{t^j}{0} + \theta^i_a(j^1_0\g) \pdeb{\g^a}{t^i}{t^j}{0} = 0 \, .
\]
Choosing a different representative $m$-curve $\ghat$ of $j^1_0\g$ which differs in its second derivatives from $\g$ (although necessarily having the same first derivatives) allows us to conclude that $\theta^i_a(j^1_0\g) = 0$, so that $\theta$ is horizontal at $j^1_0\g$ and hence is a horizontal $1$-form. We also see from this argument that
\[
\theta_a(j^1_0\g) \pde{\g^a}{t^j}{0} = 0 \, .
\]
Finally we observe that
\[
\langle \theta \, , d_k \rangle = \biggl\langle \theta_a du^a \, , u^b_k \vf{u^b} \biggr\rangle = \theta_a u^a_k
\]
so that
\[
\eval{\langle \theta \, , d_k \rangle}{j^1_0\g} = \theta_a(j^1_0\g) \pde{\g^a}{t^k}{0} = 0
\]
for each point $j^1_0\g \in T_m E$, showing that $\langle \theta \, , d_k \rangle = 0$ and hence that $\theta$ is a contact $1$-form.

The proof for forms on $\Fmb E$ is similar in principle but involves more complicated calculations.
\end{proof}
The coordinate expressions for contact $1$-forms on velocity manifolds are quite different from those on jet manifolds, and involve determinants:\ indeed, contact $1$-forms on $\Fm E$ are sums of scalar multiples of $(m+1) \times (m+1)$ determinants
\[
\theta^{a_1 a_2 \cdots a_{m+1}} = 
\begin{vmatrix}
u^{a_1}_1 & u^{a_2}_1 & \cdots & u^{a_{m+1}}_1 \\ 
u^{a_1}_2 & u^{a_2}_2 & \cdots & u^{a_{m+1}}_2 \\ 
\vdots & \vdots & & \vdots \\
u^{a_1}_m & u^{a_2}_m & \cdots & u^{a_{m+1}}_m \\ 
du^{a_1} & du^{a_2} & \cdots & du^{a_{m+1}} 
\end{vmatrix} \, .
\]
To see that such a determinant is indeed a contact $1$-form, evaluate it on the total derivative $d_k = u^b_k \p / \p u^b$ to give
\[
\langle \theta^{a_1 a_2 \cdots a_{m+1}} \, , d_k \rangle = 
\begin{vmatrix}
u^{a_1}_1 & u^{a_2}_1 & \cdots & u^{a_{m+1}}_1 \\ 
u^{a_1}_2 & u^{a_2}_2 & \cdots & u^{a_{m+1}}_2 \\ 
\vdots & \vdots & & \vdots \\
u^{a_1}_m & u^{a_2}_m & \cdots & u^{a_{m+1}}_m \\ 
u^{a_1}_k & u^{a_2}_k & \cdots & u^{a_{m+1}}_k 
\end{vmatrix}
= 0 \, .
\]
To show that these forms span the local contact 1-forms, we show that their values at each point span the contact cotangent vectors at that point. Let the coordinate functions on the fibres of $T^* \Fm E$ corresponding to the coordinates $(u^a, u^a_i)$ on $\Fm E$ be $(p_a, p_a^i)$; then horizontal cotangent vectors satisfy the equations $p_a^i = 0$, and we have seen that the condition $\langle \theta \, , d_k \rangle = 0$ corresponds to a coordinate condition which may now be written as $u^a_k p_a = 0$.

Now observe that at each point $j^1_0\g$ there is at least one set of $m$ coordinates $(u^{a_1}_1, u^{a_2}_2, \ldots, u^{a_m}_m)$ such that the determinant $\det u^{a_i}_j$ does not vanish at $j^1_0\g$; suppose, without loss of generality, that this set is $(u^1_1, u^2_2, \ldots, u^m_m)$, for we may always rearrange the order of the base coordinates $u^a$ if necessary. It is clear that the cotangent vectors
\[
\theta^{1 2 \cdots m, m+1}_{j^1_0\g}, \theta^{1 2 \cdots m, m+2}_{j^1_0\g}, \ldots,
\theta^{1 2 \cdots m, n}_{j^1_0\g}
\]
are linearly independent, so that the subspace of the space of contact cotangent vectors at $j^1_0\g$ spanned by them has dimension $n - m$. But $\dim \tau_{mE}^*(T^*_{j^1_0\g} E)= n$ and the $m$ equations $u^a_k p_a$ characterising contact $1$-forms are linearly independent for regular velocities, so that the dimension of the space of contact cotangent vectors at $j^1_0\g$ is $n - m$.


\subsection{Contact $r$-forms}

We define contact $r$-forms using the pull-back condition, so that an $r$-form $\omega$ on $\Fm E$ is a contact $r$-form if $(\jbar^1\g)^* \omega = 0$, and an $r$-form $\omega$ on $\Fmb E$ is a contact $r$-form if $(\jbar^2\g)^* \omega = 0$. Note that contact $r$-forms need not be horizontal if $r > 1$.

We now see another important difference between contact forms on velocity manifolds and contact forms on jet manifolds. In the latter context, the contact $r$-forms are generated by the contact $1$-forms and their exterior derivatives; but this is not the case on velocity manifolds. For example, on $\Fb\R^3$ the contact 1-forms are generated by the single $1$-form
\[
\theta =
\begin{vmatrix}
u^1_1 & u^2_1 & u^3_1 \\
u^1_2 & u^2_2 & u^3_2 \\
du^1 & du^2 & du^3
\end{vmatrix} \, ;
\]
but $(u^1_1 du^2 - u^2_1 du^1) \wedge du^3_2 - (u^1_2 du^2 - u^2_2 du^1) \wedge du^3_1$ is a contact 2-form which cannot be written in terms of $\theta$ and $d\theta$.


\subsection{Prolongations of maps}

Let $E_1$, $E_2$ be manifolds, and let $f : E_1 \to E_2$ a map. The \emph{prolongation} of $f$ to $T_m E_1$ is the map
\[
T_m f : T_m E_1 \to T_m E_2
\]
defined by
\[
T_m f(j^1_0\g) = j^1_0(f \circ \g) \, .
\]
It is immediate from this definition that $T_m(f \circ g) = T_m f \circ T_m g$ and that $T_m (\id_E) = \id_{T_m E}$, so that $T_m$ is a covariant functor. In coordinates,
\[
u^a \circ T_m f = f^a \, , \qquad u^a_i \circ T_m f = d_i f^a \, .
\]
It is important to note that $T_m f$ might not restrict to a map $\Fm E_1 \to \Fm E_2$, because $f\circ\g$ might not be an immersion, even though $\g$ is an immersion.


\subsection{Prolongations and the exchange map}

As a particular example, the prolongation of the vector bundle projection $\tau_{mE} : T_m E \to E$ to $T_{m^\prime} T_m E$ is
\[
T_{m^\prime} \tau_{mE} : T_{m^\prime} T_m E \to T_{m^\prime}E \, .
\]
\begin{lem}
\label{L15}
The exchange map $e : T_{m^\prime} T_m E \to T_m T_{m^\prime}E$ satisfies
\[
T_{m^\prime} \tau_{mE} \circ e = \tau_{m(T_{m^\prime}E)} \, .
\]
\end{lem}
\begin{proof}
From Lemma~\ref{L5}, $e$ may be expressed in coordinates as
\[
u^a \circ e = u^a \, , \qquad u^a_i \circ e = u^a_{;i} \, , \qquad u^a_{;j} \circ e = u^a_j \, , \qquad u^a_{i;j} \circ e = u^a_{j;i} \, .
\]
Thus
\[
u^a \circ \tau_{m(T_{m^\prime}E)} = u^a \, , \qquad u^a_i \circ \tau_{m(T_{m^\prime}E)} = u^a_i
\]
whereas
\[
u^a \circ T_{m^\prime} \tau_{mE} \circ e = u^a \circ e = u^a\, , \qquad
u^a_i \circ  T_{m^\prime} \tau_{mE} \circ e = u^a_{;i} \circ e = u^a_i \, .
\]
\end{proof}
In other words, the exchange map interchanges these two diagrams.
\begin{center}
\begin{picture}(50,70)(100,-10)
\put(0,50){\makebox(0,0){$T_{m^\prime} T_m E$}}
\put(0,0){\makebox(0,0){$T_m E$}}
\put(70,50){\makebox(0,0){$T_{m^\prime}E$}}
\put(70,0){\makebox(0,0){$E$}}
\put(25,50){\vector(1,0){25}}
\put(20,00){\vector(1,0){35}}
\put(0,40){\vector(0,-1){30}}
\put(70,40){\vector(0,-1){30}}
\put(35,55){\makebox(0,0)[b]{$\scriptstyle T_{m^\prime} \tau_{mE}$}}
\put(35,-5){\makebox(0,0)[t]{$\scriptstyle \tau_{mE}$}}
\put(75,25){\makebox(0,0)[l]{$\scriptstyle \tau_{m^\prime E}$}}
\put(-5,25){\makebox(0,0)[r]{$\scriptstyle \tau_{m^\prime (T_m E)}$}}
\put(120,30){\vector(1,0){20}}
\put(140,30){\vector(-1,0){20}}
\put(130,35){\makebox(0,0){$e$}}
\put(200,50){\makebox(0,0){$T_m T_{m^\prime} E$}}
\put(200,0){\makebox(0,0){$T_m E$}}
\put(270,50){\makebox(0,0){$T_{m^\prime}E$}}
\put(270,0){\makebox(0,0){$E$}}
\put(225,50){\vector(1,0){25}}
\put(220,00){\vector(1,0){35}}
\put(200,40){\vector(0,-1){30}}
\put(270,40){\vector(0,-1){30}}
\put(235,55){\makebox(0,0)[b]{$\scriptstyle \tau_{m(T_{m^\prime} E)}$}}
\put(235,-5){\makebox(0,0)[t]{$\scriptstyle \tau_{mE}$}}
\put(275,25){\makebox(0,0)[l]{$\scriptstyle \tau_{m^\prime E}$}}
\put(195,25){\makebox(0,0)[r]{$\scriptstyle T_m \tau_{m^\prime E}$}}
\end{picture}
\end{center}


\subsection{Prolongations of vector fields}

A vector field $X$ on $E$ is a map $E \to TE$, and so its prolongation (as a map) is $T_m X : T_m E \to T_m TE$.
\begin{lem}
The composition $X^1_m = e \circ T_m X$, where $e : T_m TE \to TT_m E$ is the exchange map, is a vector field on $T_m E$
\end{lem}
\begin{proof}
From Lemma~\ref{L15},
\begin{align*}
\tau_{T_m E} \circ e \circ T_m X
& = T_m \tau_E \circ T_m X \\
& = T_m (\tau_E \circ X) \\
& = T_m(\id_E) \\
& = \id_{T_m E} \, .
\end{align*}
\end{proof}
The vector field $X^1_m$ is called the \emph{prolongation of $X$ to $T_m E$}.
\begin{prop}
If $\psi_s$ is the flow of $X$ then $T_m \psi_s$ is the flow of $X^1_m$.
\end{prop}
\begin{proof}
We first compute a coordinate formula for the vector field whose flow is $T_m \psi_s$.

Choose a point $j^1_0\g \in T_m E$ and let $\varphi$ be the flow of $X$ in a neighbourhood of $\g(0)$. Let $(U,y)$ be a chart around $\g(0)$ so that, if
\[
X = X^a \vf{u^a} \, ,
\]
$\varphi$ satisfies
\[
\pde{\varphi^a}{s}{(0,\cdot)} = X^a \, .
\]
Let $\varphit$ denote the map $(s,q) \mapsto T_m \varphi_s(q)$, so that
\[
\varphit^a = \varphi^a \, , \qquad \varphit^a_i = d_i \varphi^a
\]
where we define $(d_i \varphi^a)(s,q) = (d_i \varphi^a_s)(q)$. Then
\begin{align*}
\pde{\varphit^a_i}{s}{(0,\cdot)} & = \pde{(d_i \varphi^a)}{s}{(0,\cdot)} = \vfe{s}{(0,\cdot)} \biggl( u^b_i \pd{\varphi^a}{u^b} \biggr) \\
& \qquad = u^b_i \pdeb{\varphi^a}{u^b}{s}{(0,\cdot)} = \eval{d_i \biggl( \pd{\varphi^a}{s}{(0,\cdot)} \biggr)}{(0,\cdot)}
= d_i X^a \, ,
\end{align*}
so that, in coordinates, the vector field whose flow is $T_m \psi_s$ is
\[
X^a \vf{u^a} + (d_i X^a) \vf{u^a_i} \, .
\]
On the other hand, regarding $X$ as a map $E \to TE$, and writing $\udot^a$ as $u^a_1$,
\[
u^a \circ X = u^a \, , \qquad u^a_1 \circ X = X^a
\]
so that
\begin{gather*}
u^a \circ T_m X  = u^a \, , \qquad u^a_1 \circ T_m X = X^a \, , \\
u^a_{;i} \circ T_m X = u^a_i \, , \qquad u^a_{1i} \circ T_m X = d_i X^a \, ;
\end{gather*}
thus
\begin{gather*}
u^a \circ e \circ T_m X  = u^a \, , \qquad u^a_{;1} \circ e \circ T_m X = X^a \, , \\
u^a_i \circ e \circ T_m X = u^a_i \, , \qquad u^a_{i1} \circ e \circ T_m X = d_i X^a
\end{gather*}
so that
\[
X^1_m = e \circ T_m X = X^a \vf{u^a} + (d_i X^a) \vf{u^a_i} \, .
\]
\end{proof}
Unlike prolongations of maps, prolongations of vector fields \emph{do} restrict to $\Fm E$.


\subsection{Second prolongations}

By extending the first order approach, maps $f : E_1 \to E_2$ may be prolonged to maps $T^2_m f : T^2_m E_1 \to T^2_m E_2$, and vector fields $X$ on $E$ may be prolonged to vector fields $X^2_m$ on $T^2_m E$. In coordinates,
\[
u^a \circ T^2_m f = f^a \, , \qquad
u^a_i \circ T^2_m f = d_i f^a \, , \qquad
u^a_{ij} \circ T^2_m f = d_i d_j f^a
\]
and if $X = X^a \p / \p u^a$ then
\[
X^2_m = X^a \vf{u^a} + (d_i X^a) \vf{u^a_i} + \frac{1}{\#(ij)} (d_i d_j X^a) \vf{u^a_{ij}} \, .
\]
The calculations are similar in principle to those given for the first order case, but more complicated in detail. Again $T^2_m f$ might not restrict to a map $\Fmb E_1 \to \Fmb E_2$, whereas $X^2_m$ does restrict to $\Fmb E$.


\subsection{Prolongations, contact forms, and total derivatives}

Let $f : E_1 \to E_2$ be a map. If $\theta$ is a contact form on $\Fm E_2$ and if $T_m f$ restricts to $\Fm E_1$ then $(T_m f)^* \theta$ is a contact form on $\Fm E_1$, because
\[
(\jbar^1\g)^* (T_m f)^* \theta = (T_m f \circ \jbar^1\g)^* \theta = \bigl( \jbar^1(f \circ \g) \bigr)^* \theta = 0 \, .
\]
If $X$ is a vector field on $E$ and $\theta$ is a contact form on $\Fm E$ then the Lie derivative $\Lie_{X^1_m} \theta$ by the prolongation of $X$ is also a contact form, because the flow of $X^1_m$ is the prolongation of the flow of $X$. These results, using the characterisation of a contact form by vanishing pullback, apply to both $1$-forms and to $r$-forms with $r>1$. They also hold for contact forms on $\Fmb E$.

The corresponding result for total derivatives is more complicated, as these operators are vector fields along a map rather than on a manifold.
\begin{lem}
Prolongations and basis total derivatives commute, so that
\[
d_i \circ \Lie_X = \Lie_{X^1_m} \circ d_i \, , \qquad
d_i \circ \Lie_{X^1_m} = \Lie_{X^2_m} \circ d_i \, .
\]
\end{lem}
\begin{proof}
We check this using coordinates. In the first order case, if $f$ is a function on $E$ then
\[
d_i(\Lie_X f) = d_i \biggl( X^a \pd{f}{u^a} \biggr) = u^b_i \biggl( \pd{X^a}{u^b} \pd{f}{u^a} + X^a \pdb{f}{u^b}{u^a} \biggr)
\]
whereas
\[
\Lie_{X^1_m} (d_i f) = \Lie_{X^1_m} \biggl( u^b_i \pd{f}{u^b} \biggr)
= (d_i X^b) \pd{f}{u^b} + u^b_i X^a \pdb{f}{u^b}{u^a} \, .
\]
A similar but slightly more lengthy calculation is used in the second order case.
\end{proof}


\subsection{Vertical endomorphisms}

We have seen that $T_m E \to E$ is a vector bundle and so, as with every vector bundle, it has a canonical vertical lift operator. Denote the vertical lift to $(\eta_i) \in \bigoplus^m TE \cong T_m E$ by 
\[
T_{m | \tau_m(\eta_i)}E \to T_{(\eta_i)} T_m E \, , \qquad 
(\xi_k) \mapsto (\xi_k)^{\uparrow(\eta_i)} \, ;
\]
in coordinates this is
\[
(\xi_k)^{\uparrow(\eta_i)} = u^a_j(\xi_k) \vfe{u^a_j}{(\eta_i)} \, . 
\]
For each vector $\zeta \in T_{(\eta_i)} T_m E$ define the vector $S^j \zeta \in T_{(\eta_i)} T_m E$ by
\[
S^j \zeta = (0, \ldots, 0, T\tau_m(\zeta), 0, \ldots, 0)^{\uparrow (\eta_i)}
\]
where the non-zero vector $T\tau_m(\zeta)$ is in the $j$-th position. It is evident that $S^j$ is a vector bundle map $TT_m E \to TT_m E$, or alternatively a type $(1,1)$ tensor field on $T_m E$, called a \emph{vertical endomorphism}. Note that the superscript $j$ is a counting index, not a coordinate index. In coordinates
\[
S^j = du^a \otimes \vf{u^a_j} \, .
\]
There is a close relationship between vertical endomorphisms and total derivatives.
\begin{lem}
\label{L19}
If $\omega$ is an $r$-form on $E$ then
\[
S^j d_k \omega = r \d^j_k (\tau_{mE}^* \omega) \, .
\]
\end{lem}
\begin{proof}
Suppose first that $\theta$ is a $1$-form; we shall give a proof in coordinates, omitting explicit mention of the pullback map. If $\theta = \theta_a du^a$ then
\[
S^j d_k \theta = S^j \bigl( (d_k \theta_a) du^a + \theta_a du^a_k \bigr) = \d^j_k \theta_a du^a = \d^j_k \theta \, .
\]
We now use induction on $r$. Suppose $\omega$ is an $r$-form and that $S^j d_k \omega = r \d^j_k (\tau_{mE}^* \omega)$; then
\begin{align*}
S^j d_k (\theta \wedge \omega) & = S^j \bigl( d_k \theta \wedge \tau_{mE}^* \omega + \tau_{mE}^* \theta \wedge d_k \omega \bigr) \\
& = S^j d_k \theta \wedge \tau_{mE}^* \omega + \tau_{mE}^* \theta \wedge S^j d_k \omega \\
& = \d^j_k (\tau_{mE}^* \theta \wedge \tau_{mE}^* \omega) + r \d^j_k (\tau_{mE}^* \theta \wedge \tau_{mE}^* \omega) \\
& = (r+1) \d^j_k \, \tau_{mE}^* (\theta \wedge \omega)
\end{align*}
using the fact that $\tau_{mE}^* \theta$ and $\tau_{mE}^* \omega$ are horizontal over $E$. The result now follows by linearity.
\end{proof}


\subsection{Second order vertical endomorphisms}

There is also a version of the vertical endomorphism defined on second order velocity manifolds. This cannot be constructed in the same way as the first order vertical endomorphism, as $T^2_m E \to T_m E$ is not a vector bundle but is instead an affine sub-bundle of $T_m T_m E \to T_m E$. We shall establish our construction by modifying the first-order vertical endomorphism on $T_m T_m E$. There is an alternative method, based on the construction of vertical lifts using double $(1,m)$-curves, which may be used in both first and second order cases, but we shall not describe that here.

So let $\nu : T^2_m E \to T_m E$ be some tubular neighbourhood of $T^2_m E$ in $T_m T_m E$, and let $\iota : T^2_m E \to T_m T_m E$ be the inclusion from Proposition~\ref{P4}. As before, let $e : T_m T_m E \to T_m T_m E$ be the exchange map.
\begin{prop}
Let $\theta$ be a $1$-form on $T^2_m E$; then the operation
\[
\theta \mapsto \iota^* \bigl( S^k \hook (\nu^* \theta + e^* \nu^* \theta) \bigr) \, ,
\]
where $S^k$ is the vertical endomorphism on $T_m T_m E$), does not depend on the choice of tubular neighbourhood map $\nu$ and hence defines a vertical endomorphism on $T^2_m E$.
\end{prop}
\begin{proof}
We use coordinates to show that the result is independent of $\nu$. Let $\theta = \theta_a du^a + \theta^i_a du^a_i + \theta^{ij}_a du^a_{ij}$; then
\begin{align*}
\nu^*\theta & = (\nu^* \theta_a) d\nu^a + (\nu^* \theta^i_a) d\nu^a_i + (\nu^* \theta^{ij}_a) d\nu^a_{ij} \\
& = (\nu^* \theta_a) \biggl( du^a + \pd{\nu^a}{u^c_p} (du^c_p - du^c_{;p}) 
+ \tfrac{1}{2} \pd{\nu^a}{u^c_{p;q}} (du^c_{p;q} - du^c_{q;p}) \biggr) \\*
& \qquad + (\nu^* \theta^i_a) \biggl( \tfrac{1}{2} (du^a_i + du^a_{;i})
+ \tfrac{1}{2} \biggl( \pd{\nu^a_i}{u^c_p} - \pd{\nu^a_i}{u^c_{;p}} \biggr) (du^c_p - du^c_{;p}) \\*
& \qquad\qquad + \tfrac{1}{2} \pd{\nu^a_i}{u^c_{p;q}} (du^c_{p;q} - du^c_{q;p}) \biggr) \\*
& \qquad + (\nu^* \theta^{ij}_a) \biggl( \tfrac{1}{2} (du^a_{i;j} +  du^a_{j;i}) + \pd{\nu^a_{ij}}{u^c_p} (du^c_p - du^c_{;p})
+ \tfrac{1}{2} \pd{\nu^a_{ij}}{u^c_{p;q}} (du^c_{p;q} - du^c_{q;p}) \biggr)
\end{align*}
using the coordinate formul\ae\ for the tubular neighbourhood map given in Section~2. Thus
\begin{align*}
S^k \hook \nu^* \theta
& = (\nu^* \theta_a) \biggl( - \pd{\nu^a}{u^c_k} du^c 
+ \tfrac{1}{2} \biggl( \pd{\nu^a}{u^c_{p;k}} - \pd{\nu^a}{u^c_{k;p}} \biggr) du^c_p \biggr) \\*
& \qquad + (\nu^* \theta^i_a) \biggl( \d^k_i \tfrac{1}{2} du^a
- \tfrac{1}{2} \biggl( \pd{\nu^a_i}{u^c_k} - \pd{\nu^a_i}{u^c_{;k}} \biggr) du^c \\*
& \qquad\qquad + \tfrac{1}{2} \biggl( \pd{\nu^a_i}{u^c_{p;k}} - \pd{\nu^a_i}{u^c_{k;p}} \biggr) du^c_p \biggr) \\*
& \qquad + (\nu^* \theta^{ij}_a) \biggl( \tfrac{1}{2} (\d^k_j du^a_i +  \d^k_i du^a_j) - \pd{\nu^a_{ij}}{u^c_k} du^c \\*
& \qquad\qquad + \tfrac{1}{2} \biggl( \pd{\nu^a_{ij}}{u^c_{p;k}} - \pd{\nu^a_{ij}}{u^c_{k;p}} \biggr) du^c_p \biggr)
\end{align*}
so that
\begin{align*}
\iota^* (S^k \hook \nu^* \theta)
& = \theta_a \biggl( - \iota^*\biggl( \pd{\nu^a}{u^c_k} \biggr) du^c 
+ \tfrac{1}{2} \iota^*\biggl( \pd{\nu^a}{u^c_{p;k}} - \pd{\nu^a}{u^c_{k;p}} \biggr) du^c_p \biggr) \\*
& \qquad + \theta^i_a \biggl( \d^k_i \tfrac{1}{2} du^a
- \tfrac{1}{2} \iota^*\biggl( \pd{\nu^a_i}{u^c_k} - \pd{\nu^a_i}{u^c_{;k}} \biggr) du^c \\*
& \qquad\qquad + \tfrac{1}{2} \iota^*\biggl( \pd{\nu^a_i}{u^c_{p;k}} - \pd{\nu^a_i}{u^c_{k;p}} \biggr) du^c_p \biggr) \\*
& \qquad + \theta^{ij}_a \biggl( \tfrac{1}{2} (\d^k_j du^a_i +  \d^k_i du^a_j) - \iota^* \biggl( \pd{\nu^a_{ij}}{u^c_k} \biggr) du^c \\*
& \qquad\qquad + \tfrac{1}{2} \iota^*\biggl( \pd{\nu^a_{ij}}{u^c_{p;k}} - \pd{\nu^a_{ij}}{u^c_{k;p}} \biggr) du^c_p \biggr) \, ;
\end{align*}
and similarly
\begin{align*}
S^k \hook e^*\nu^*\theta
& = (e^*\nu^* \theta_a) \biggl( e^* \biggl( \pd{\nu^a}{u^c_k} \biggr) du^c 
+ \tfrac{1}{2} e^*\biggl( \pd{\nu^a}{u^c_{k;p}} - \pd{\nu^a}{u^c_{p;k}} \biggr) du^c_p \biggr) \\*
& \qquad + (e^*\nu^* \theta^i_a) \biggl( \tfrac{1}{2} \d^k_i du^a
+ \tfrac{1}{2} e^*\biggl( \pd{\nu^a_i}{u^c_k} - \pd{\nu^a_i}{u^c_{;k}} \biggr) du^c \\*
& \qquad\qquad + \tfrac{1}{2} e^*\biggl( \pd{\nu^a_i}{u^c_{k;p}} - \pd{\nu^a_i}{u^c_{p;k}} \biggr) du^c_p \biggr) \\*
& \qquad + (e^*\nu^* \theta^{ij}_a) \biggl( \tfrac{1}{2} (\d^k_j du^a_i +  \d^k_i du^a_j) 
+ e^*\biggl( \pd{\nu^a_{ij}}{u^c_k} \biggr) du^c \\*
& \qquad\qquad + \tfrac{1}{2} e^*\biggl( \pd{\nu^a_{ij}}{u^c_{k;p}} - \pd{\nu^a_{ij}}{u^c_{p;k}} \biggr) du^c_p \biggr)
\end{align*}
so that, using $\iota^* e^* = \iota^*$,
\begin{align*}
\iota^*(S^k \hook e^*\nu^*\theta)
& = \theta_a \biggl( \iota^* \biggl( \pd{\nu^a}{u^c_k} \biggr) du^c 
+ \tfrac{1}{2} \iota^*\biggl( \pd{\nu^a}{u^c_{k;p}} - \pd{\nu^a}{u^c_{p;k}} \biggr) du^c_p \biggr) \\*
& \qquad + \theta^i_a \biggl( \tfrac{1}{2} \d^k_i du^a
+ \tfrac{1}{2} \iota^*\biggl( \pd{\nu^a_i}{u^c_k} - \pd{\nu^a_i}{u^c_{;k}} \biggr) du^c \\*
& \qquad\qquad + \tfrac{1}{2} \iota^*\biggl( \pd{\nu^a_i}{u^c_{k;p}} - \pd{\nu^a_i}{u^c_{p;k}} \biggr) du^c_p \biggr) \\*
& \qquad + \theta^{ij}_a \biggl( \tfrac{1}{2} (\d^k_j du^a_i +  \d^k_i du^a_j) 
+ \iota^*\biggl( \pd{\nu^a_{ij}}{u^c_k} \biggr) du^c \\*
& \qquad\qquad + \tfrac{1}{2} \iota^*\biggl( \pd{\nu^a_{ij}}{u^c_{k;p}} - \pd{\nu^a_{ij}}{u^c_{p;k}} \biggr) du^c_p \biggr) \, .
\end{align*}
Thus, adding, we obtain
\[
\iota^* (S^k \hook (\nu^* \theta + e^* \nu^* \theta))
= \theta^k_a du^a + 2 \theta^{ik}_a du^a_i
\]
using $\theta^{ki}_a = \theta^{ik}_a$.
\end{proof}
In coordinates, therefore, the second order vertical endomorphisms may be written as tensor fields
\[
S^k = du^a \otimes \vf{u^a_k} + \frac{2}{\#(ik)} du^a_i \otimes \vf{u^a_{ik}} \, ;
\]
the factor $1 / \#(ik)$ arises here because the contraction of $\p / \p u^a_{ik}$ with $du^c_{pq}$ equals $\frac{1}{2} \#(ik) \d^c_a (\d^i_p \d^k_q + \d^i_q \d^k_p)$, so that
\[
\vf{u^a_{ik}} \hook (\theta^{pq}_c du^c_{pq}) = \frac{\#(ik)}{2} \d^c_a (\d^i_p \d^k_q + \d^i_q \d^k_p) \theta^{pq}_c
= \#(ik) \, \theta^{ik}_a \, .
\]
The relationship given in Lemma~\ref{L19} between vertical endomorphisms and total derivatives may now be extended to a kind of homotopy formula.
\begin{lem}
\label{L21}
If $\omega$ is an $r$-form on $T_m E$ then
\[
S^j d_k \omega - d_k S^j \omega = r \, \d^j_k (\tau_{mE}^{2,1 \, *} \omega) \, .
\]
\end{lem}
\begin{proof}
Suppose first that $\theta$ is a $1$-form; we shall give a proof in coordinates, omitting explicit mention of the pullback map. If $\theta = \theta_a du^a + \theta^i_a du^a_i$ then
\[
d_k \theta = (d_k \theta_a) du^a + \theta_a du^a_k + (d_k \theta^i_a) du^a_i + \theta^i_a du^a_{ik}
\]
so that
\[
S^j d_k \theta = \bigl( \d^j_k \theta_a + (d_k \theta^j_a) \bigr) du^a + \d^j_k \theta^i_a  du^a_i + \theta^j_a  du^a_k \, .
\]
On the other hand, $S^j \theta = \theta^j_a du^a$, so that
\[
d_k S^j \theta = (d_k \theta^j_a) du^a + \theta^j_a du^a_k
\]
and hence
\[
S^j d_k \theta - d_k S^j \theta = \d^j_k \theta_a du^a + \d^j_k \theta^i_a  du^a_i = \d^j_k \theta \, .
\]
We now use induction on $r$. Suppose $\omega$ is an $r$-form and that $S^j d_k \omega - d_k S^j \omega = r \d^j_k (\tau_{mE}^{2,1 \, *} \omega)$; then, as both $S^j$ and $d_k$ are derivations of degree zero, their commutator is a derivation of degree zero, and so
\begin{align*}
(S^j d_k - d_k S^j)(\theta \wedge \omega) & = (S^j d_k - d_k S^j) \theta \wedge \tau_{mE}^{2,1 \, *} \omega
+ \tau_{mE}^{2,1 \, *} \theta \wedge (S^j d_k - d_k S^j) \omega \\
& = r \d^j_k \, \tau_{mE}^{2,1 \, *}\theta \wedge \tau_{mE}^{2,1 \, *} \omega 
+ \d^j_k \tau_{mE}^{2,1 \, *}\theta \wedge \tau_{mE}^{2,1 \, *} \omega \\
& = (r+1) \d^j_k \tau_{mE}^{2,1 \, *}(\theta \wedge \omega) \, .
\end{align*}
The result now follows by linearity.
\end{proof}


\section{Vector forms}

We often use \emph{vectors} of operators, tensors, forms, and so on. For instance, we have defined the total derivatives $d_k$ and the vertical endomorphisms $S^j$, where $j$ and $k$ are counting indices rather than coordinate indices. These operators fit into a framework of \emph{vector forms}, to which we can associate a cohomology theory. Although the full cohomology theory requires the use of higher-order velocity manifolds, we can see some aspects of the theory in the first and second order cases.


\subsection{Vector forms}

We consider differential forms on $E$, $\Fm E$ and $\Fmb E$ taking values in the vector space $\R^{m*}$ and its exterior powers. Write $\Fmk E$ with $k=0,1,2$ and put
\[
\Omega^{r,s}_k = \left( \Omega^r \Fmk E \right) \otimes \left( \bw^s \R^{m*} \right) \, .
\]
Then a typical element of $\Omega^{r,s}_k$ is
\[
\Xi = \chi_{i_1 \cdots i_s} \otimes dt^{i_1} \wedge \ldots \wedge dt^{i_s} \in \Omega^{r,s}_k 
\]
where the scalar forms $\chi_{i_1 \cdots i_s}$ are skew-symmetric in their indices, and where, as in Corollary~\ref{C2}, $\{dt^i\}$ is the canonical basis of $\R^{m*}$. It is clear that $\Omega^{r,s}_k$ is a module over the algebra of functions on $\Fmk E$.


\subsection{Operations on vector forms}

Define the operators $d$ and $\dT$ on the modules of vector forms by their actions on decomposable forms,
\begin{align*}
d : \Omega^{r,s}_k \to \Omega^{r+1,s}_k \, , \qquad
& d(\chi \otimes \omega) = d\chi \otimes \omega \\
\dT : \Omega^{r,s}_k \to \Omega^{r,s+1}_{k+1} \, , \qquad
& \dT(\chi \otimes \omega) = d_i \chi \otimes (dt^i \wedge \omega) \, ,
\end{align*}
so that
\begin{align*}
d\dT(\chi\otimes\omega) & = d \bigl( d_i \chi \otimes (dt^i \wedge \omega) \bigr) = dd_i \chi \otimes (dt^i \wedge \omega) \\
& \qquad \qquad = d_i d\chi \otimes (dt^i \wedge \omega) = \dT(d\chi \otimes \omega) = \dT d(\chi\otimes\omega)
\end{align*}
and
\[
\dT^2(\chi\otimes\omega) = d_j d_i \chi \otimes (dt^j \wedge dt^i \wedge \omega) = 0 \, ,
\]
showing that $d \dT = \dT d$ and $\dT^2 = 0$. We say that $d\Xi$ is the \emph{differential} of the vector form $\Xi$, and that $\dT\Xi$ is its \emph{total derivative}.

The total derivative of a vector form is a type of Lie derivative, and so we can also define the corresponding contraction operation. Put
\[
\iT : \Omega^{r,s}_k \to \Omega^{r-1,s+1}_{k+1} \, , \qquad
\iT(\chi \otimes \omega) = (d_i \hook \chi) \otimes dt^i \wedge \omega
\]
where $d_i \hook \chi$ denotes the contraction of the `vector field along a map' $d_i$ with the scalar form $\chi$, so that
\[
\dT = d \iT + \iT d \, .
\] 


\subsection{Equivariant vector forms}

Let $\alpha_{j^1_0\phi} : \Fm E \to \Fm E$ denote the right action of $j^1_0\phi \in L^{1+}_m$ on $\Fm E$ by
\[
\alpha_{j^1_0\phi}(j^1_0\g) = j^1_0(\g\circ\phi) \, ;
\]
also, let $A_{j^1_0\phi} : \R^{m*} \to \R^{m*}$ denote the linear map
\[
A_{j^1_0\phi}(dt^i) = \bigl( D_j \phi^i(0)\bigr)  dt^j \, ,
\]
and extend this by multilinearity to $A_{j^1_0\phi} : \bw^s \R^{m*} \to \bw^s \R^{m*}$. The vector form $\chi_{i_1 \cdots i_s} \otimes (dt^{i_1} \wedge \cdots \wedge dt^{i_s}) \in \Omega^{r,s}_1$ is said to be \emph{equivariant} if, for every $j^1_0\phi$,
\[
\alpha_{j^1_0\phi}^*(\chi_{i_1 \cdots i_s}) \otimes (dt^{i_1} \wedge \cdots \wedge dt^{i_s})
= \chi_{i_1 \cdots i_s} \otimes A_{j^1_0\phi} (dt^{i_1} \wedge \cdots \wedge dt^{i_s}) \, .
\]
Thus an equivariant form, regarded as a map from objects defined on a velocity manifold to elements of a vector space, commutes with the action of the jet group on the manifold and the vector space. We use the \emph{oriented} jet group in our definition, as our application will be to problems in the calculus of variations where we need to integrate the forms.

We shall be particularly interested in equivariant elements of $\Omega^{0,m}_1$, namely $0$-forms (functions) taking their values in the one-dimensional vector space $\bw^m \R^{m*}$. Then
\[
A_{j^1_0\phi} (dt^1 \wedge \cdots \wedge dt^m) = \J\phi(0) (dt^1 \wedge \cdots \wedge dt^m)
\]
where $\J\phi = \det (D_j \phi^i)$ is the Jacobian of $\phi$, and so, writing $d^m t$ for $dt^1 \wedge \cdots \wedge dt^m$, an element $\Lambda = L \, d^m t$ is equivariant when
\[
(L \circ \alpha_{j^1_0\phi}) d^m t = \det \bigl( D_j \phi^i(0) \bigr) L \, d^m t \, .
\]
Thus, writing an element of $T_m E \cong \bigoplus^m TE$ as $(\xi_1, \ldots, \xi_m)$, $\Lambda$ is equivariant when for each matrix $A \in \GL^+(m,\R)$,
\[
L(\xi_i A^i_j) = (\det A) L(\xi_j) \, .
\]

As the oriented jet group $L^{1+}_m$ is connected, there is an infinitesimal condition for equivariance. For a vector form $\chi_{i_1 \cdots i_s} \otimes (dt^{i_1} \wedge \cdots \wedge dt^{i_s}) \in \Omega^{r,s}_1$, we require
\[
d^j_i (\chi_{i_1 \cdots i_s}) \otimes (dt^{i_1} \wedge \cdots \wedge dt^{i_s})
= \chi_{i_1 \cdots i_s} \otimes \Lie_{t^j \p/\p t^i}(dt^{i_1} \wedge \cdots \wedge dt^{i_s})
\]
In the particular case where $s=m$ we have $\Lie_{t^j \p/\p t^i} d^m t = \d^j_i d^m t$, so the condition simplifies to
\[
d^j_i \chi = \d^j_i \chi \, .
\]


\subsection{The bicomplex}

It is clear that for $-1 \le s \le m-2$ we can use the operators $d$ and $\dT$ to construct a bicomplex:
\begin{center}
\begin{picture}(300,120)(50,-20)
\multiput(98,70)(60,0){4}{\vector(0,-1){20}}
\multiput(98,30)(60,0){4}{\vector(0,-1){20}}
\multiput(70,80)(0,-40){3}{\vector(1,0){20}}
\multiput(130,80)(0,-40){3}{\vector(1,0){20}}
\multiput(190,80)(0,-40){3}{\vector(1,0){20}}
\multiput(250,80)(0,-40){3}{\vector(1,0){20}}
\multiput(310,80)(0,-40){3}{\vector(1,0){20}}
\multiput(60,80)(0,-40){3}{\makebox(0,0){$0$}}
\put(105,80){\makebox(0,0){$\overline{\Omega}^{0,s}_0$}}
\put(165,80){\makebox(0,0){$\Omega^{1,s}_0$}}
\put(225,80){\makebox(0,0){$\Omega^{2,s}_0$}}
\put(285,80){\makebox(0,0){$\Omega^{3,s}_0$}}
\put(110,40){\makebox(0,0){$\overline{\Omega}^{0,s+1}_1$}}
\put(170,40){\makebox(0,0){$\Omega^{1,s+1}_1$}}
\put(230,40){\makebox(0,0){$\Omega^{2,s+1}_1$}}
\put(290,40){\makebox(0,0){$\Omega^{3,s+1}_1$}}
\put(110,0){\makebox(0,0){$\overline{\Omega}^{0,s+2}_2$}}
\put(170,0){\makebox(0,0){$\Omega^{1,s+2}_2$}}
\put(230,0){\makebox(0,0){$\Omega^{2,s+2}_2$}}
\put(290,0){\makebox(0,0){$\Omega^{3,s+2}_2$}}
\multiput(135,85)(0,-40){3}{\makebox(0,0)[b]{$\scriptstyle d$}}
\multiput(200,85)(0,-40){3}{\makebox(0,0)[b]{$\scriptstyle d$}}
\multiput(255,85)(0,-40){3}{\makebox(0,0)[b]{$\scriptstyle d$}}
\multiput(95,60)(0,-40){2}{\makebox(0,0)[r]{$\scriptstyle \dT$}}
\multiput(155,60)(0,-40){2}{\makebox(0,0)[r]{$\scriptstyle \dT$}}
\multiput(215,60)(0,-40){2}{\makebox(0,0)[r]{$\scriptstyle \dT$}}
\multiput(275,60)(0,-40){2}{\makebox(0,0)[r]{$\scriptstyle \dT$}}
\end{picture}
\end{center}
where if $s=-1$ then $\Omega^{*,s}_* = 0$. In this bicomplex $\overline{\Omega}^{0,*}_*$ means `modulo constant functions', and is used instead of the usual beginning $0 \to \R \to \Omega^0 \to \ldots$ of the de Rham sequence.

An important property of the bicomplex is that all columns (apart from the first) are globally exact, we show this by obtaining a homotopy formula for $\dT$. Strictly speaking the homotopy formula involves \emph{third order} forms which are horizontal over $E$, because the operator $P_2$ defined in the statement of the theorem involves applying a total derivative to (scalar) second-order forms which are horizontal over $E$; but if $\dT\Xi = 0$ then the operator $P_2$ is not involved and the formula is genuinely second order. We feel, nevertheless, that it is worthwhile giving the more general statement, on the understanding that the definition of the total derivative of a second order form, and the consequent generalisation of Lemma~\ref{L21}, follow exactly the same pattern as before. We also use the operator $P_2$ when studying equivalents of first-order Lagrangians, although in that context the image of $P_2$ is always second-order rather than third-order.
\begin{theorem}
\label{T22}
If $\Xi \in \Omega^{r,s+1}_1$ with $r>0$ then, to within a pullback,
\[
P_2 \dT \Xi + \dT P_1 \Xi = \Xi \, ,
\]
where
\[
P_1 \bigl( \chi_{i_1 \cdots i_{s+1}} \otimes dt^{i_1} \wedge \cdots \wedge dt^{i_{s+1}} \bigr)
= \frac{1}{r(m-s)} S^j \chi_{i_1 \cdots i_{s+1}} \otimes \biggl( \vf{t^j} \hook dt^{i_1} \wedge \cdots \wedge dt^{i_{s+1}} \biggr)
\]
for first-order $r$-forms $\chi_{i_1 \cdots i_{s+1}}$, and
\begin{align*}
\lefteqn{P_2 \bigl( \eta_{i_1 \cdots i_{s+2}} \otimes dt^{i_1} \wedge \cdots \wedge dt^{i_{s+2}} \bigr)} \\
& = \biggl( \frac{1}{r(m-s-1)} S^j  \eta_{i_1 \cdots i_{s+2}}
- \frac{1}{r^2(m-s)(m-s-1)}d_l S^l S^j \eta_{i_1 \cdots i_{s+2}} \biggr) \otimes \\*
& \qquad \otimes \biggl( \vf{t^j} \hook dt^{i_1} \wedge \cdots \wedge dt^{i_{s+1}} \biggr)
\end{align*}
for second-order $r$-forms $\eta_{i_1 \cdots i_{s+2}}$.
\end{theorem}
\begin{proof}
This is a consequence of Lemma~\ref{L21}. Put
\begin{align*}
P^j_1 & = \frac{1}{r(m-s)} S^j \\
P^j_2 & = \frac{1}{r(m-s-1)} S^j - \frac{1}{r^2(m-s)(m-s-1)}d_l S^l S^j \, ;
\end{align*}
then
\begin{align*}
P_2 \dT \Xi 
& = P_2 \bigl( d_k\chi_{i_1 \cdots i_{s+1}} \otimes  dt^k \wedge dt^{i_1} \wedge \cdots \wedge dt^{i_{s+1}} \bigr) \\
& = P^j_2 d_k \chi_{i_1 \cdots i_{s+1}} 
\otimes \biggl( \vf{t^j} \hook dt^k \wedge dt^{i_1} \wedge \cdots \wedge dt^{i_{s+1}} \biggr) \\
& = P^j_2 d_k \chi_{i_1 \cdots i_{s+1}} \otimes 
(\d^k_j \, dt^{i_1} \wedge \cdots \wedge dt^{i_{s+1}}) \\*
& \qquad - P^j_2 d_k \chi_{i_1 \cdots i_{s+1}} \otimes 
dt^k \wedge \biggl( \vf{t^j} \hook dt^{i_1} \wedge \cdots \wedge dt^{i_{s+1}} \biggr) \\
& = P^k_2 d_k \chi_{i_1 \cdots i_{s+1}} \otimes 
dt^{i_1} \wedge \cdots \wedge dt^{i_{s+1}} \\*
& \qquad - (s+1)P^j_2 d_{i_1} \chi_{j i_2 \cdots i_{s+1}} \otimes 
dt^{i_1} \wedge dt^{i_2} \wedge \cdots \wedge dt^{i_{s+1}}
\end{align*}
whereas
\begin{align*}
\dT P_1 \Xi
& = \dT \biggl( P^j_1 \chi_{i_1 \cdots i_{s+1}} 
\otimes \biggl( \vf{t^j} \hook dt^{i_1} \wedge \cdots \wedge dt^{i_{s+1}} \biggr) \biggr) \\
& = d_k P^j_1 \chi_{i_1 \cdots i_{s+1}} 
\otimes dt^k \wedge \biggl( \vf{t^j} \hook dt^{i_1} \wedge \cdots \wedge dt^{i_{s+1}} \biggr) \\
& = (s+1) d_{i_1} P^j_1 \chi_{j i_2 \cdots i_{s+1}} 
\otimes dt^{i_1} \wedge dt^{i_2} \wedge \cdots \wedge dt^{i_{s+1}}
\end{align*}
so that
\begin{align*}
P_2 \dT \Xi + \dT P_1 \Xi 
& = P^k_2 d_k \chi_{i_1 \cdots i_{s+1}} \otimes 
dt^{i_1} \wedge \cdots \wedge dt^{i_{s+1}} \\*
& \qquad - (s+1)P^j_2 d_{i_1} \chi_{j i_2 \cdots i_{s+1}} \otimes 
dt^{i_1} \wedge \cdots \wedge dt^{i_{s+1}} \\*
& \qquad + (s+1)d_{i_1} P^j_1 \chi_{j i_2 \cdots i_{s+1}} 
\otimes dt^{i_1} \wedge \cdots \wedge dt^{i_{s+1}} \, .
\end{align*}
But, using Lemma~\ref{L21}, the operators acting on $\chi_{j i_2 \cdots i_{s+1}}$ satisfy
\begin{align*}
\d^j_{i_1} P^k_2 d_k
& = \frac{1}{r(m-s-1)} \d^j_{i_1} S^k d_k - \frac{1}{r^2(m-s)(m-s-1)} \d^j_{i_1} d_l S^l S^k d_k \\
& = \frac{1}{r(m-s-1)} \d^j_{i_1} (d_k S^k + mr) \\*
& \qquad - \frac{1}{r^2(m-s)(m-s-1)} \d^j_{i_1} (d_l d_k S^l S^k + (m+1)r d_l S^l) \\
& = \frac{m}{m-s-1} \d^j_{i_1} - \frac{s+1}{r(m-s)(m-s-1)} \d^j_{i_1} d_k S^k \, ,
\end{align*}
using the fact that $S^l S^k \chi_{i_1 \cdots i_{s+1}} = 0$ because the $\chi_{i_1 \cdots i_{s+1}}$ are first-order forms. Similarly
\begin{align*}
- (s+1)P^j_2 d_{i_1}
& = - \frac{s+1}{r(m-s-1)} S^j d_{i_1} + \frac{s+1}{r^2(m-s)(m-s-1)} d_l S^l S^j d_{i_1} \\
& = - \frac{s+1}{r(m-s-1)} (d_{i_1} S^j + r \d^j_{i_1}) \\*
& \qquad + \frac{s+1}{r^2(m-s)(m-s-1)} (d_l d_{i_1} S^l S^j + r d_{i_1} S^j + r \d^j_{i_1} d_l S^l ) \\
& = - \frac{s+1}{(m-s-1)} \d^j_{i_1} - \frac{(s+1)}{r(m-s)} d_{i_1} S^j + \frac{s+1}{r(m-s)(m-s-1)} \d^j_{i_1} d_k S^k
\end{align*}
and
\[
(s+1)d_{i_1} P^j_1 = \frac{s+1}{r(m-s)} d_{i_1} S^j \, ,
\]
from which we see that
\[
\d^j_{i_1} P^k_2 d_k - (s+1)P^j_2 d_{i_1} + (s+1)d_{i_1} P^j_1 = \d^j_{i_1}
\]
and the result follows.
\end{proof}


\subsection{The bottom left corner}

The part of the bicomplex which holds the major interest for the calculus of variations is in the bottom left-hand corner; we shall repeat it, with a pull-back map shown explicitly where appropriate.
\begin{center}
\begin{picture}(100,90)(0,-20)
\put(20,0){\makebox(0,0){$\overline{\Omega}^{0,m}_1$}}
\put(80,0){\makebox(0,0){$\Omega^{1,m}_1$}}
\put(140,0){\makebox(0,0){$\Omega^{1,m}_2$}}
\put(145,50){\makebox(0,0){$\Omega^{1,m-1}_1$}}
\put(35,0){\vector(1,0){30}}
\put(95,0){\vector(1,0){30}}
\put(132,40){\vector(0,-1){30}}
\put(95,10){\vector(1,1){30}}
\put(50,-5){\makebox(0,0)[t]{$\scriptstyle d$}}
\put(110,-5){\makebox(0,0)[t]{$\scriptstyle \tau_{mE}^{2,1\,*}$}}
\put(105,25){\makebox(0,0)[r]{$\scriptstyle S$}}
\put(137,25){\makebox(0,0)[l]{$\scriptstyle \dT$}}
\end{picture}
\end{center}
Take $[\Lambda]\in\overline{\Omega}^{0,m}_1$, so that, for some function $L$ on $\Fm E$, we have for any representative
\[
\Lambda = L \, dt^1 \wedge \cdots \wedge dt^m = L \, d^m t \, .
\]
Here, $L$ will play the role of a (first order) Lagrangian function in the calculus of variations, and the vector-valued function $\Lambda$ will have the capability of being integrated along $m$-curves in $\Fm E$ (and, in particular, along prolongations to $\Fm E$ of $m$-curves in $E$). So, given the equivalence class $[\Lambda]$, define
\[
\Theta_1 = Sd\Lambda \, , \qquad \ce_0 = \tau_{mE}^{2,1\,*} d\Lambda - \dT \Theta_1 \, ,
\]
where the choice of representative in the equivalence class is immaterial as we consider only $d\Lambda$ in the definition. We may compute $\Theta_1$ and $\ce_0$ in coordinates; they are
\begin{align*}
\Theta_1 & = S^j \biggl( \pd{L}{u^a} du^a + \pd{L}{u^a_i} du^a_i \biggr) 
\otimes \biggl( \vf{t^j} \hook (dt^1 \wedge \cdots \wedge dt^m) \biggr) \\
& = \biggl( \pd{L}{u^a_j} du^a \biggr) \otimes \biggl( \vf{t^j} \hook (dt^1 \wedge \cdots \wedge dt^m) \biggr)
\end{align*}
and
\begin{align*}
\ce_0 & = \biggl( \pd{L}{u^a} du^a + \pd{L}{u^a_i} du^a_i \biggr) \otimes (dt^1 \wedge \cdots \wedge dt^m) \\*
& \qquad - d_k \biggl( \pd{L}{u^a_j} du^a \biggr) \otimes dt^k \wedge \biggl( \vf{t^j} \hook (dt^1 \wedge \cdots \wedge dt^m) \biggr) \\
& = \biggl( \pd{L}{u^a} - d_k \biggl(\pd{L}{u^a_k} \biggr) \biggr) du^a \otimes (dt^1 \wedge \cdots \wedge dt^m)  \, .
\end{align*}


\section{Variational problems}

Our main application of the theory of vector forms, and their associated cohomology, will be to problems in the calculus of variations. These will be \emph{parametric problems}:\ that is, problems where the solutions are submanifolds without a given parametrization (although with a particular orientation). In the one-dimensional case, as exemplified by Finsler geometry, all the vector forms are essentially scalar forms, and so this theory only provides further insight in the case where the submanifolds have dimension two or more.


\subsection{Homogeneous variational problems}

We now study $m$-dimensional variational problems on $E$, with fixed boundary conditions. As before, a vector function $\L = L \, d^m t \in \Omega_1^{0,m}$ will be called a \emph{Lagrangian} for a variational problem. It will be called \emph{homogeneous} if it is equivariant with respect to the action of the oriented jet group $L^{1+}_m$. Thus $\L$ is homogeneous when the scalar function $L$ satisfies the infinitesimal condition
\[
d^i_j L = \delta^i_j L
\]
or, equivalently, the finite condition
\[
L \circ \alpha_{j^1_0\phi} = (\det D_j \phi^i(0)) L
\]
for every every $j^1_0\phi \in L^{1+}_m$.

We now consider submanifolds of $E$ of the form $\g(C)$ where $\g : \R^m \to E$ is an immersion and $C \subset \R^m$ is a connected compact $m$-dimensional submanifold with boundary $\p C$. The \emph{fixed-boundary variational problem defined by $\Lambda$} is the search for extremal submanifolds $\g(C) \subset E$ satisfying
\[
\int_C ((\jbar^1\g)^* \Lie_{X^1_m} L) d^m t = 0  
\]
for every variation field $X$ on $E$ satisfying $\eval{X}{\g(\p C)} = 0$.
\begin{theorem}
If $\Lambda$ is homogeneous and $\g(C)$ is an extremal submanifold then $\g \circ \phi$ is also an extremal submanifold, for any orientation-preserving reparametrization $\phi$ whose image contains $C$.
\end{theorem}
\begin{proof}
We shall show that if $\Lambda$ is homogeneous then, for any immersion $\g$,
\[
\int_{\phi^{-1}(C)} \bigl( L \circ \jbar^1(\g\circ\phi) \bigr) d^m t = \int_C \bigl( L \circ \jbar^1\g \bigr) d^m t
\]
so that the integral itself is invariant under reparametrization; hence extremals will be invariant under reparametrization. As
\begin{align*}
\int_{\phi^{-1}(C)} \bigl( L \circ \jbar^1(\g\circ\phi) \bigr) d^m t 
& = \int_C (\phi^{-1})^* \bigl( \bigl( L \circ \jbar^1(\g\circ\phi) \bigr) d^m t \bigr) \\
& = \int_C \bigl( L \circ \jbar^1(\g\circ\phi) \circ \phi^{-1} \bigr) \, (\phi^{-1})^* d^m t \, ,
\end{align*}
it will be sufficient to show that
\[
\bigl( L \circ \jbar^1(\g\circ\phi) \circ \phi^{-1} \bigr) \, (\phi^{-1})^* d^m t = \bigl( L \circ \jbar^1\g \bigr) d^m t \, .
\]
Now for any $s\in\R^m$
\[
\eval{d^m t}{s} = (\J\phi \circ \phi^{-1})(s) \eval{(\phi^{-1})^* d^m t}{s} \, ,
\]
and so it will be sufficient to show that, for each $s$,
\[
\bigl( L \circ \jbar^1(\g\circ\phi) \circ \phi^{-1} \bigr)(s) = \bigl( L \circ \jbar^1\g \bigr)(s) (\J\phi \circ \phi^{-1})(s) \, .
\]
Note that we do not require the diffeomorphism $\phi$ to satisfy the condition $\phi(0) = 0$.

To see how this can be obtained from the homogeneity condition, write the latter as
\[
L \circ \alpha_{j^1_0\varphi} = (\J\varphi)(0) L
\]
where $\varphi$ is a diffeomorphism which \emph{does} satisfy $\varphi(0) = 0$; then, for any immersion $\g : \R^m \to E$,
\begin{align*}
(\J\varphi)(0) L \bigl( j^1_0(\g\circ\tr_s) \bigr) & = L \bigl( \alpha_{j^1_0\varphi} \bigl(  j^1_0(\g\circ\tr_s) \bigr) \bigr) \\
& = L \bigl( j^1_0(\g\circ\tr_s \circ\varphi) \bigr) \, .
\end{align*}
Now put $\varphi = \tr_{-s} \circ \phi \circ \tr_{\phi^{-1}(s)}$, and note that $\varphi(0) = 0$; also
\[
(\g \circ \tr_s) \circ \varphi = \g \circ \phi \circ \tr_{\phi^{-1}(s)} 
\]
and
\[
(\J\varphi)(0) = (\J\phi)(\phi^{-1}(s)) \, ,
\]
so that
\[
(\J\phi)(\phi^{-1}(s)) L \bigl( j^1_0(\g\circ\tr_s) \bigr)
= L \bigl( j^1_0(\g \circ \phi \circ \tr_{\phi^{-1}(s)}) \bigr)
\]
and hence
\[
(\J\phi)(\phi^{-1}(s)) L \bigl( \jbar^1\g(s) \bigr)
= L \bigl( \jbar^1(\g \circ \phi) \circ \phi^{-1}(s) \bigr) \, .
\]
\end{proof}


\subsection{Equivalents of Lagrangians}

Let $\L \in \Omega^{0,m}_1$ be a homogeneous Lagrangian. Any scalar $m$-form $\Theta_m \in \Omega^{m,0}_1$ which is horizontal over $E$ will be called an \emph{integral equivalent} of $\L$ if
\[
\L = \bigg( \frac{(-1)^{m(m-1)/2}}{m!} \bigg) \iT^m \Theta_m \, ;
\]
any vector $r$-form $\Theta_r \in \Omega^{r,m-r}_1$ which is horizontal over $E$ will be called an \emph{intermediate equivalent} if
\[
\L = \frac{(-1)^{r(r-1)/2}(m-r)!}{m!} \, \iT^r \Theta_r \qquad 0 \le r \le m-1 \, .
\]
\begin{lem}
\label{L24}
If $\Theta_{r+1}$ is an equivalent of $\L$ then
\[
\Theta_r = \frac{(-1)^r}{m-r} \, \iT \Theta_{r+1}
\]
is also an equivalent. 
\end{lem}
\begin{proof}
If $\Theta_{r+1}$ is an equivalent of $\L$ then by definition
\[
\L = \frac{(-1)^{r(r+1)/2}(m-r-1)!}{m!} \, \iT^{r+1} \Theta_{r+1} \, ,
\]
so that
\begin{align*}
\frac{(-1)^{r(r-1)/2}(m-r)!}{m!} \, \iT^r \Theta_r
& = \frac{(-1)^{r(r-1)/2}(m-r)!}{m!} \, \iT^r \biggl( \frac{(-1)^r}{m-r} \, \iT \Theta_{r+1} \biggr) \\
& = \L \, .
\end{align*}
\end{proof}
In the case $r=m$ we use the term `integral equivalent' for the following reason.
\begin{lem}
\label{L25}
If $\g$ is an $m$-curve in $E$ then $(\jbar^1\g)^* \L = (\jbar^1\g)^* \Theta_m$, so that
\[
\int_C (\jbar^1\g)^* \L = \int_C (\jbar^1\g)^* \Theta_m \, .
\]
It follows that $\L = \Theta_0$ and $\Theta_m$ have the same extremals.
\end{lem}
\begin{proof}
Suppose $\Theta \in \Omega^{r,m-r}$ may be written in coordinates in the particular form
\[
\Theta = \Theta_{a_1 \cdots a_m} u^{a_{r+1}}_{k_{r+1}} \cdots u^{a_m}_{k^m} du^{a_1} \wedge \cdots \wedge du^{a_r} 
\otimes dt^{k_{r+1}} \wedge \cdots \wedge dt^{k_m}
\]
where the functions $\Theta_{a_1 \cdots a_m}$ are skew-symmetric in their indices; then
\begin{align*}
\iT\Theta & = \Theta_{a_1 \cdots a_m} u^{a_{r+1}}_{k_{r+1}} \cdots u^{a_m}_{k^m}
\biggl( u^b_{k_r} \vf{u^b} \hook du^{a_1} \wedge \cdots \wedge du^{a_r} \biggr) \otimes \\*
& \qquad\qquad\qquad \otimes dt^{k_r} \wedge dt^{k_{r+1}} \wedge \cdots \wedge dt^{k_m} \\
& = \sum_{p=1}^r (-1)^{p-1}\Theta_{a_1 \cdots a_m} u^{a_{r+1}}_{k_{r+1}} \cdots u^{a_m}_{k^m}
\biggl( u^{a_p}_{k_r} du^{a_1} \wedge \cdots \widehat{du^{a_p}} \cdots \wedge du^{a_r} \biggr) \otimes \\*
& \qquad\qquad\qquad \otimes dt^{k_r} \wedge dt^{k_{r+1}} \wedge \cdots \wedge dt^{k_m} \\
& = r (-1)^{r-1}\Theta_{a_1 \cdots a_m} u^{a_r}_{k_r} u^{a_{r+1}}_{k_{r+1}} \cdots u^{a_m}_{k^m}
du^{a_1} \wedge \cdots \wedge du^{a_{r-1}} \otimes \\*
& \qquad\qquad\qquad \otimes dt^{k_r} \wedge dt^{k_{r+1}} \wedge \cdots \wedge dt^{k_m} \, .
\end{align*}
Thus if $\Theta \in \Omega^{m,0}$ we see that
\begin{align*}
\iT^m \Theta & = m! (-1)^{m(m-1)/2} \Theta_{a_1 \cdots a_m} u^{a_1}_{k_1} \cdots u^{a_m}_{k^m} dt^{k_1} \wedge \cdots \wedge dt^{k_m} \\
& = m! (-1)^{m(m-1)/2} \Theta_{a_1 \cdots a_m} \det \bigl( u^{a_i}_{k_j}\bigr) dt^1 \wedge \cdots \wedge dt^m
\end{align*}
so that
\[
(\jbar^1\g)^* \iT^m \Theta = m! (-1)^{m(m-1)/2} \bigl( \Theta_{a_1 \cdots a_m} \circ \jbar^1\g \bigr) 
\det \biggl( \pd{\g^{a_i}}{t^{k_j}} \biggr) dt^1 \wedge \cdots \wedge dt^m \, .
\]
On the other hand,
\begin{align*}
(\jbar^1\g)^* \Theta & = \bigl( \Theta_{a_1 \cdots a_m} \circ \jbar^1\g \bigr)
(\jbar^1\g)^* \bigl( du^{a_1} \wedge \cdots \wedge du^{a_m} \bigr) \\
& = \bigl( \Theta_{a_1 \cdots a_m} \circ \jbar^1\g \bigr) \det \biggl( \pd{\g^{a_i}}{t^{k_j}} \biggr) dt^1 \wedge \cdots \wedge dt^m \, .
\end{align*}
\end{proof}


\subsection{Euler forms}

Let $\Theta_m$ be an integral equivalent of $\L$. Define the scalar $(m+1)$-form $\ce_m \in \Omega^{m+1,0}_1$ by
\[
\ce_m = d\Theta_m
\]
and the vector forms $\ce_r \in \Omega^{r+1,m-r}_2$ by
\[
\ce_r = \tau_{mE}^{2,1\,*} d\Theta_r - (-1)^r \dT\Theta_{r+1} \qquad 0 \le r \le m-1 \, .
\]
The forms $\ce_r$ are called the \emph{Euler forms} of $\Theta_m$.
\begin{lem}
\label{L26}
The Euler forms satisfy the recurrence relation
\[
\ce_r = \frac{(-1)^{r+1}}{m-r} \, \iT \ce_{r+1} \qquad 0 \le r \le m-1 \, ;
\]
consequently if $d\Theta_m = \ce_m = 0$ then $\ce = 0$.
\end{lem}
\begin{proof}
This follows from the definition and Lemma~\ref{L24}. We have, omitting the pull-back maps,
\begin{align*}
\iT\ce_{r+1} & = \iT d\Theta_{r+1}  (-1)^{r+1} \iT\dT\Theta_{r+2} \\
& = \dT\Theta_{r+1} - d\iT\Theta_{r+1} - (-1)^r \dT\iT\Theta_{r+2} \\
& = \dT\Theta_{r+1} - (-1)^r d\Theta_r + (m-r-1) \dT\Theta_{r+1} \\
& = (m-r) \bigl( \dT\Theta_{r+1} - (-1)^r d\Theta_r \bigr)
\intertext{when $r+1<m$, so that}
\frac{(-1)^{r+1}}{m-r} \iT\ce_{r+1} & = (-1)^{r+1} \dT\Theta_{r+1} + d\Theta^r = \ce_r \, .
\end{align*}
Similarly,
\begin{align*}
\iT\ce_m & = \iT d\Theta_m \\
& = \dT\Theta_m - d\iT\Theta_m \\
& = \dT\Theta_m - (-1)^{m-1} d\Theta_{m-1}
\intertext{so that}
(-1)^m \iT\ce_m & = (-1)^m \dT\Theta_m + d\Theta_{m-1} = \ce_{m-1} \, .
\end{align*}
\end{proof}
The different spaces containing the various equivalents and Euler forms may be seen in this diagonal part of the bicomplex.
\begin{center}
\begin{picture}(300,210)(50,-30)
\put(270,150){\vector(0,-1){20}}
\put(238,110){\vector(0,-1){20}}
\put(178,70){\vector(0,-1){20}}
\put(118,30){\vector(0,-1){20}}
\put(80,0){\vector(1,0){30}}
\put(150,40){\vector(1,0){20}}
\put(210,80){\vector(1,0){20}}
\put(292,160){\vector(1,0){20}}
\put(245,158){\makebox(0,0){\emph{$\Theta_m \in$}}}
\put(275,160){\makebox(0,0){$\Omega^{m,0}_1$}}
\put(335,160){\makebox(0,0){$\Omega^{m+1,0}_1$}}
\put(367,158){\makebox(0,0){\emph{$\ni\ce_m$}}}
\put(260,120){\makebox(0,0){$\dots$}}
\put(158,79){\makebox(0,0){\emph{$\Theta_2 \in$}}}
\put(190,80){\makebox(0,0){$\Omega^{2,m-2}_1$}}
\put(98,39){\makebox(0,0){\emph{$\Theta_1 \in$}}}
\put(250,80){\makebox(0,0){$\Omega^{3,m-2}_2$}}
\put(280,79){\makebox(0,0){\emph{$\ni\ce_2$}}}
\put(40,-2){\makebox(0,0){\emph{$\L \in$}}}
\put(130,40){\makebox(0,0){$\Omega^{1,m-1}_1$}}
\put(190,40){\makebox(0,0){$\Omega^{2,m-1}_2$}}
\put(220,39){\makebox(0,0){\emph{$\ni\ce_1$}}}
\put(65,0){\makebox(0,0){$\overline{\Omega}^{0,m}_1$}}
\put(125,0){\makebox(0,0){$\Omega^{1,m}_2$}}
\put(150,-1){\makebox(0,0){\emph{$\ni\ce_0$}}}
\put(95,5){\makebox(0,0)[b]{$\scriptstyle \dt$}}
\put(160,45){\makebox(0,0)[b]{$\scriptstyle \dt$}}
\put(220,85){\makebox(0,0)[b]{$\scriptstyle \dt$}}
\put(125,20){\makebox(0,0)[l]{$\scriptstyle \dT$}}
\put(185,60){\makebox(0,0)[l]{$\scriptstyle \dT$}}
\put(245,100){\makebox(0,0)[l]{$\scriptstyle \dT$}}
\put(277,140){\makebox(0,0)[l]{$\scriptstyle \dT$}}
\put(302,165){\makebox(0,0)[b]{$\scriptstyle d$}}
\put(110,30){\vector(-3,-2){30}}
\put(170,70){\vector(-3,-2){30}}
\put(230,110){\vector(-3,-2){30}}
\put(170,30){\vector(-3,-2){30}}
\put(230,70){\vector(-3,-2){30}}
\put(290,110){\vector(-3,-2){30}}
\put(312,150){\vector(-3,-2){30}}
\put(252,150){\vector(-3,-2){15}}
\put(110,120){\makebox(0,0){\emph{$\displaystyle\Theta_r = \frac{(-1)^r}{m-r} \, \iT \Theta_{r+1}, \quad
\Theta_0 = \L$}}}
\put(300,20){\makebox(0,0){\emph{$\ce_r = \dt\Theta_r - (-1)^r \dT\Theta_{r+1}$}}}
\put(270,4){\makebox(0,0){\emph{$\dt = \tau_{mE}^{2,1\,*} \circ d$}}}
\end{picture}
\end{center}


\subsection{Lepagian forms}

Let $\L$ be a homogeneous Lagrangian, and let $\Theta_r$ be an equivalent of $\L$ ($1 \le r \le m$). We shall say that $\Theta_r$ is \emph{Lepagian} if the corresponding Euler form $\ce_0 \in \Omega^{1,m}_2$ satisfies
\[
S\ce_0 = 0 \, ,
\]
so that $\ce_0$ is horizontal over $E$.
\begin{theorem}
The vector 1-form
\[
\Theta_1 = Sd\L
\]
is an integral equivalent of $\L$ ($m=1$) or an intermediate equivalent ($m \ge 2$), and is Lepagian. It is called the \emph{Hilbert equivalent} of $\L = L \, d^m t$.
\end{theorem}
\begin{proof}
From the definition of $S$,
\[
S\Xi = S^j \chi \otimes d^{m-1} t_j \, ,
\]
so that
\begin{align*}
\iT S d\L & = \iT S(dL \otimes d^m t) \\
& = \iT(S^j dL \otimes d^{m-1} t_j) \\
& = i_k S^j dL \otimes dt^k \wedge d^{m-1} t_j \\
& = i_j S^j dL \otimes d^m t \, .
\end{align*}
But for any $1$-form $\theta$ on $\Fm E$, if in coordinates $\theta = \theta_a du^a + \theta^i_a du^a_i$ then
\[
i_j S^j \theta = i_j (\theta^j_a du^a) = u^a_j \theta^j_a \, ,
\]
so that
\[
i_j S^j dL = u^a_j \pd{L}{u^a_j} = d^j_j L = mL
\]
using the homogeneity of the Lagrangian.

To show that $\Theta_1$ is Lepagian, note that
\begin{align*}
S \dT \Theta_1 & = S \dT S d\L \\
& = S \dT (S^j dL \otimes d^{m-1} t_j) \\ 
& = S \big( d_i S^j dL \otimes (dt^i \wedge d^{m-1} t_j) \big) \\
& = S \big( d_j S^j dL \otimes d^m t \big) \\
& = S^i d_j S^j dL \otimes d^{m-1} t_i \\
& = (d_j S^i + \d^i_j) S^j dL \otimes d^{m-1} t_i  \\
& = S^i dL \otimes d^{m-1} t_i \\
& = S(dL \otimes d^m t) = Sd\L 
\end{align*}
using Lemma~\ref{L21} and the fact that $L$ is defined on $\Fm E$ so that $S^i S^j dL = 0$; thus $S\ce_0 = 0$, as required.
\end{proof}
\begin{theorem}
If $\Tht_1$ is another Lepagian vector 1-form equivalent to $\L$, with corresponding Euler form $\cet_0$, then
\[
\cet_0 = \ce_0 \, , \qquad \Tht_1 - \Theta_1 = \dT\Phi \qquad (\Phi \in \Omega^{1,m-2}_0) \, .
\]
\end{theorem}
\begin{proof}
It follows straightforwardly from the Lepagian condition $S\cet_0 = 0$ that $P_2\cet_0 = 0$, so that we may use the homotopy condition of Theorem~\ref{T22} to see that
\[
0 = P_2\cet_0 = P_2(d\L - \dT\Tht_1) = \Theta_1 - P_2\dT\Tht_1 = \Theta_1 - (1 - \dT P_1)\Tht_1 \, ,
\]
giving $\Tht_1 - \Theta_1 = \dT P_1 \Tht_1$ (or $\Tht_1 = \Theta_1$ if $m=1$). Thus
\[
\cet_0 - \ce_0 = (d\L - \dT\Tht_1) - (d\L - \dT\Theta_1) = - \dT^2 P_1 \Tht_1 = 0 \, .
\]
(Note that, as $d\Lambda$ is a first-order vector $1$-form, $P_2 d\Lambda = Sd\Lambda = \Theta_1$.)
\end{proof}


\subsection{The First Variation Formula}

\begin{theorem}
Let $C$ be a compact connected $m$-dimensional submanifold of $\R^m$ with boundary $\p C$, let $\g$ be an $m$-curve in $E$ whose domain contains $C$, and let $X$ be a variation field on $E$ vanishing on $\g(\p C)$ with prolongation $X^1_m$ on $\Fm E$. Then
\[
\int_C (\jbar^1\g)^* \Lie_{X^1_m} \L = \int_C (\jbar^2\g)^* i_X \ce_0 \, ;
\]
consequently $\g$ is an extremal of $\L$ precisely when $\ce_0$ vanishes along the image of $\jbar^2\g$.
\end{theorem}
\begin{proof}
We note first that
\begin{align*}
\int_C (\jbar^1\g)^* \Lie_{X^1_m} \L & = \int_C (\jbar^1\g)^* i_{X^1_m} d\L \\
& = \int_C (\jbar^2\g)^* i_{X^2_m} \tau_{mE}^{2,1\,*}d\L \\
& = \int_C (\jbar^2\g)^* i_{X^2_m} \ce_0 + \int_C (\jbar^2\g)^* i_{X^2_m} \dT \Theta_1 \, ,
\end{align*}
using the definition of the Euler form $\ce_0$. But prolongations commute with basis total derivatives and $\Theta_1$ is horizontal over $E$, so that
\[
\int_C (\jbar^2\g)^* i_{X^2_m} \dT \Theta_1 = \int_C (\jbar^2\g)^* \dT i_{X^1_m} \Theta_1 = \int_C d (\jbar^1\g)^* i_X \Theta_1 = 0
\]
and we see that the second integral vanishes; thus
\[
\int_C (\jbar^1\g)^* \Lie_{X^1_m} \L = \int_C (\jbar^2\g)^* i_{X^2_m} \ce_0 = \int_C (\jbar^2\g)^* i_X \ce_0
\]
because $\ce_0$ is horizontal over $E$.

Now let $\g$ be an immersion. If $\ce_0=0$ at every point in the image of $\jbar^2\g$, then for any vector field $X$ on $E$ and any $t \in C$ we will have $\eval{(\jbar^2\g)^* i_X \ce_0}{t}=0$, so that the integral over $C$ will vanish and $\g$ will be an extremal.

If, instead, $q = j^2_0(\g\circ\tr_t)$ is some point in the image of $\jbar^2\g$ where $\eval{\ce_0}{q}$ is non-zero, then there must be a vector field $X$ on $E$ such that the vector-valued function $i_X \ce_0$ gives a strictly positive multiple of $d^m t$ when evaluated at $q$, and hence when evaluated in some neighbourhood $U$ of $q$. Let $b$ be a positive bump function on $E$ whose support lies in the interior of $U$ and which satisfies $b(q) = 1$. Then
\[
\int_C (\jbar^1\g)^* \Lie_{(bX)^1_m} \L = \int_C (\jbar^2\g)^* i_{bX} \ce_0 > 0 \, ,
\]
so that $\g$ cannot be an extremal.
\end{proof}


\subsection{Integral equivalents for $m \ge 2$}

Let $\L = L \, d^m t$ be a homogeneous Lagrangian with $m \ge 2$, and write its Hilbert equivalent $\Theta_1$ as
\[
\Theta_1 = \vt^i \otimes d^{m-1} t_i \, ;
\]
the scalar 1-forms $\vt_i$ are called the \emph{Hilbert forms} of $\L$. If $\L$ never vanishes, define the \emph{\Cth\ equivalent} $\Tht_m \in \Omega^{m,0}_1$ by
\[
\Tht_m = \frac{1}{L^{m-1}} \bigwedge_{i=1}^m \vt^i \, .
\]
\begin{theorem}
The \Cth\ equivalent $\Tht_m$ is an integral equivalent of $\L$.
\end{theorem}
\begin{proof}
We must show that $\iT^m \Theta_m = (-1)^{m(m-1)/2} \, m!\L$, so rewrite $\Theta_m$ as
\[
\Theta_m = \frac{1}{m! L^{m-1}} \sum_{\sigma \in \Sg_m} (-1)^\sigma \vt^{\sigma(1)} \wedge \cdots \wedge \vt^{\sigma(m)} \, ,
\]
where $\Sg_m$ is the permutation group, and use induction. The calculation uses $d_j \hook \vt^i = \d^i_j L$, the proof of which is similar to that used to show that $\iT \Theta_1 = m\L$; we also define $\tau_{r,s} \in \Sg_m$ by
\[
\tau_{r,s}(i) =
\begin{cases}
m-s & (i = r) \\
i-1 & (r+1 \le i \le m-s) \\
i & \text{otherwise} \, .
\end{cases}
\]
Now
\begin{align*}
\lefteqn{\iT \bigg( \sum_{\sigma \in \Sg_m} (-1)^\sigma 
\vt^{\sigma(1)} \wedge \cdots \wedge \vt^{\sigma(m-s)}
\otimes dt^{\sigma(m-s+1)} \wedge \cdots \wedge dt^{\sigma(m)} \bigg)} \\
& = \sum_{\sigma \in \Sg_m} (-1)^\sigma 
d_j \hook \big( \vt^{\sigma(1)} \wedge \cdots \wedge \vt^{\sigma(m-s)} \big)
\otimes dt^j \wedge dt^{\sigma(m-s+1)} \wedge \cdots \wedge dt^{\sigma(m)} \\
& = \sum_{r=1}^{m-s} \sum_{\sigma \in \Sg_m} (-1)^\sigma (-1)^{r-1}
\big( \vt^{\sigma(1)} \wedge \cdots \wedge (d_j \hook \vt^{\sigma(r)}) \wedge \cdots \wedge \vt^{\sigma(m-s)} \big) \otimes \\*
& \qquad\qquad \otimes dt^j \wedge dt^{\sigma(m-s+1)} \wedge \cdots \wedge dt^{\sigma(m)} \\
& = L \sum_{r=1}^{m-s} \sum_{\sigma \in \Sg_m} (-1)^\sigma (-1)^{r-1}
\big( \vt^{\sigma(1)} \wedge \cdots \wedge \vt^{\sigma(r-1)} \wedge \vt^{\sigma(r+1)}
\wedge \cdots \wedge \vt^{\sigma(m-s)} \big) \otimes \\*
& \qquad\qquad \otimes dt^{\sigma(r)} \wedge dt^{\sigma(m-s+1)} \wedge \cdots \wedge dt^{\sigma(m)} \\
& = L \sum_{r=1}^{m-s} \sum_{\sigma \in \Sg_m} (-1)^\sigma (-1)^{r-1} (-1)^{m-r-s} \bigg\{ \\*
& \qquad \big( \vt^{\sigma\tau_{r,s}(1)} \wedge \cdots \wedge \vt^{\sigma\tau_{r,s}(r-1)} \wedge \vt^{\sigma\tau_{r,s}(r+1)}
\wedge \cdots \wedge \vt^{\sigma\tau_{r,s}(m-s)} \big) \otimes \\*
& \qquad\qquad \otimes dt^{\sigma\tau_{r,s}(r)} \wedge dt^{\sigma\tau_{r,s}(m-s+1)} 
\wedge \cdots \wedge dt^{\sigma\tau_{r,s}(m)} \bigg\} \\ 
& = (-1)^{m-s-1}L \sum_{r=1}^{m-s} \sum_{\sigma \in \Sg_m} (-1)^\sigma 
\big( \vt^{\sigma(1)} \wedge \cdots \wedge \vt^{\sigma(m-s-1)} \big) \otimes \\*
& \qquad\qquad \otimes dt^{\sigma(m-s)} \wedge dt^{\sigma(m-s+1)} \wedge \cdots \wedge dt^{\sigma(m)} \\
& = (-1)^{m-s-1}(m-s)L \sum_{\sigma \in \Sg_m} (-1)^\sigma 
\big( \vt^{\sigma(1)} \wedge \cdots \wedge \vt^{\sigma(m-s-1)} \big) \otimes \\*
& \qquad\qquad \otimes dt^{\sigma(m-s)} \wedge dt^{\sigma(m-s+1)} \wedge \cdots \wedge dt^{\sigma(m)} \, ,
\end{align*}
so if
\begin{align*}
\iT^s \Theta_m & = \frac{(-1)^{s(2m-s-1)/2}}{(m-s)!L^{m-s-1}} \bigg\{ \\*
& \qquad \sum_{\sigma \in \Sg_m} (-1)^\sigma 
\vt^{\sigma(1)} \wedge \cdots \wedge \vt^{\sigma(m-s)}
\otimes dt^{\sigma(m-s+1)} \wedge \cdots \wedge dt^{\sigma(m)} \bigg\}
\end{align*}
then
\begin{align*}
\iT^{s+1} \Theta_m & = \frac{(-1)^{s(2m-s-1)/2}}{(m-s)!L^{m-s-1}} \bigg\{ \\*
& \qquad  (-1)^{m-s-1}(m-s)L \sum_{\sigma \in \Sg_m} (-1)^\sigma 
\big( \vt^{\sigma(1)} \wedge \cdots \wedge \vt^{\sigma(m-s-1)} \big) \otimes \\*
& \qquad\qquad \otimes dt^{\sigma(m-s)} \wedge dt^{\sigma(m-s+1)} \wedge \cdots \wedge dt^{\sigma(m)} \bigg\} \\
& = \frac{(-1)^{(s+1)(2m-s-2)/2}}{(m-s-1)!L^{m-s-2}} \sum_{\sigma \in \Sg_m} (-1)^\sigma 
\big( \vt^{\sigma(1)} \wedge \cdots \wedge \vt^{\sigma(m-s-1)} \big) \otimes \\*
& \qquad\qquad \otimes dt^{\sigma(m-s)} \wedge dt^{\sigma(m-s+1)} \wedge \cdots \wedge dt^{\sigma(m)} 
\end{align*}
as required. Hence
\begin{align*}
\iT^m \Theta_m & = \frac{(-1)^{m(m-1)/2}}{L^{-1}}
\sum_{\sigma \in \Sg_m} (-1)^\sigma dt^{\sigma(1)} \wedge \cdots \wedge dt^{\sigma(m)} \\
& = (-1)^{m(m-1)/2} \, m! L \, dt^1 \wedge \cdots \wedge dt^m \\
& = (-1)^{m(m-1)/2} \, m! \L \, .
\end{align*}
\end{proof}
We see also from the induction formula that
\[
\iT^{m-1} \Theta_m = (-1)^{m(m-1)/2} \, (m-1)! \Theta_1
\]
where $\Theta_1$ is the Hilbert equivalent; consequently $\Theta_m$ is Lepagian. Then, as $d\Theta_m = \ce_m$,
\begin{align*}
\int_C (j^1\g)^* \Lie_{X^1_m} \Theta_m & = \int_C (j^1\g)^* i_{X^1_m} \ce_m \\
& = (-1)^{m(m-1)/2}m! \int_C (j^2\g)^* i_X \ce_0 
\end{align*}
for any vector field $X$ on $E$ vanishing on $\g(\p C)$, because contractions by vector fields anticommute, so that $\iT^m i_{X^1_m} \ce_m = (-1)^m i_{X^1_m} \iT^m \ce_m$.


\subsection{Another integral equivalent}

When $m=1$ then the only Lepagian integral equivalent of a Lagrangian is the Hilbert equivalent. But when $m>1$ there may be other integral equivalents. Put
\[
\Theta_{r+1} = \frac{(-1)^r}{(r+1)^2} Sd\Theta_r \qquad (1 \le r < m)
\]
where, as usual, $\Theta_0 = \Lambda$.
\begin{lem}
Each $\Theta_r$ is a first-order vector form, an element of $\Omega^{r,m-r}_1$, horizontal over $E$.
\end{lem}
\begin{proof}
Each $\Theta_r$ is first-order because neither $S$ nor $d$ increases the order of a vector form. By definition $\Theta_0$ is horizontal over $E$, and if $\Theta_r$ is horizontal over $E$ then the contraction of $d\Theta_r$ with any vector field on $\Fm E$ vertical over $E$ will again be horizontal over $E$; thus $\Theta_{r+1}$ will also be horizontal over $E$.
\end{proof}
\begin{theorem}
The scalar $m$-form $\Theta_m$ is a Lepagian integral equivalent of $\L$ called the \emph{fundamental equivalent} of $\L$.
\end{theorem}
\begin{proof}
We first show that, in coordinates,
\[
\Theta_r = \frac{1}{(r!)^2} \pdx{r}{L}{u^{a_1}_{i_1}}{u^{a_r}_{i_r}} du^{a_1} \wedge \cdots \wedge du^{a_r}
\otimes \biggl( \vf{t^{i_r}} \hook \cdots \hook \vf{t^1} \hook d^m t \biggr) \, .
\]
This formula clearly holds for $r=1$ (and, indeed, for $r=0$); so suppose that it holds for a given value of $r$. Then
\begin{align*}
\Theta_{r+1} & = \frac{(-1)^r}{(r+1)^2} Sd\Theta_r \\
& = \frac{(-1)^r}{(r+1)^2} \frac{1}{(r!)^2} S^j \biggl( \pdxa{r+1}{L}{u^{a_1}_{i_1}}{u^{a_r}_{i_r}}{u^{a_{r+1}}_{i_{r+1}}} du^{a_{r+1}}_{i_{r+1}} + \cdots \biggr) \wedge \\*
& \qquad \wedge du^{a_1} \wedge \cdots \wedge du^{a_r} 
\otimes \biggl( \vf{t^j} \hook \vf{t^{i_r}} \hook \cdots \hook \vf{t^1} \hook d^m t \biggr) \\
& = \frac{1}{((r+1)!)^2} \pdx{r+1}{L}{u^{a_1}_{i_1}}{u^{a_{r+1}}_{i_{r+1}}}
du^{a_1} \wedge \cdots \wedge du^{a_{r+1}} \otimes \\*
& \qquad \otimes \biggl( \vf{t^{i_{r+1}}} \hook \cdots \hook \vf{t^1} \hook d^m t \biggr)
\end{align*}
so that the formula also holds for the case $r+1$. In particular, therefore, we have
\begin{align*}
\Theta_m & = \frac{1}{(m!)^2} \pdx{m}{L}{u^{a_1}_{i_1}}{u^{a_m}_{i_m}} du^{a_1} \wedge \cdots \wedge du^{a_m}
\times \biggl( \vf{t^{i_m}} \hook \cdots \hook \vf{t^1} \hook d^m t \biggr) \\
& = \frac{1}{(m!)^2} \pdx{m}{L}{u^{a_1}_{i_1}}{u^{a_m}_{i_m}} du^{a_1} \wedge \cdots \wedge du^{a_m} \times
\begin{vmatrix}
\d^1_{i_1} & \cdots & \d^1_{i_m} \\
\vdots & & \vdots \\
\d^m_{i_1} & \cdots & \d^m_{i_m}
\end{vmatrix} \\
& = \frac{1}{m!} \pdx{m}{L}{u^{a_1}_1}{u^{a_m}_m} du^{a_1} \wedge \cdots \wedge du^{a_m} \, .
\end{align*}
Thus, using the calculation in the proof of Lemma~\ref{L25},
\begin{align*}
\iT^m \Theta_m & = m! (-1)^{m(m-1)/2} \biggl( \frac{1}{m!} \pdx{m}{L}{u^{a_1}_1}{u^{a_m}_m}
\det \bigl( u^{a_i}_{k_j} \bigr) \biggr) dt^1 \wedge \cdots \wedge dt^m \\
& = (-1)^{m(m-1)/2} \pdx{m}{L}{u^{a_1}_1}{u^{a_m}_m} \det \bigl( u^{a_i}_{k_j} \bigr) dt^1 \wedge \cdots \wedge dt^m \\
& = (-1)^{m(m-1)/2} m! \, L \, dt^1 \wedge \cdots \wedge dt^m \\
& = (-1)^{m(m-1)/2} m! \, \Lambda \, .
\end{align*}
\end{proof}
\begin{theorem}
The fundamental equivalent $\Theta_m$ of a homogeneous Lagrangian $\Lambda$ has the property that $d\Theta_m = \ce_m = 0$ if, and only if, $\ce_0 = 0$.
\end{theorem}
\begin{proof}
If $\ce_m = 0$ then $\ce_0 = 0$ by the recurrence relation of Lemma~\ref{L26}. So show the converse, we use the definition
\[
\Theta_{r+1} = \frac{(-1)^r}{(r+1)^2} Sd\Theta_r 
\]
and the fact that $d\Theta_r \in \Omega^{r+1,m-r}_1$ to see that the homotopy operator $P_1$ from Theorem~\ref{T22} takes the form
\[
P_1 \bigl( \chi_{i_1 \cdots i_{m-r}} \otimes dt^{i_1} \wedge \cdots \wedge dt^{i_{m-r}} \bigr)
= \frac{1}{(r+1)^2} S^j \chi_{i_1 \cdots i_{m-r}} \otimes \biggl( \vf{t^j} \hook dt^{i_1} \wedge \cdots \wedge dt^{i_{m-r}} \biggr)
\]
(the formula in the proof of Theorem~\ref{T22} was for an element of $\Omega^{r,s+1}_1$); thus we may rewrite the definition of $\Theta_{r+1}$ as
\[
\Theta_{r+1} = (-1)^r Pd\Theta_r \, .
\]
Now from
\[
\ce_r = d\Theta_r - (-1)^r \dT\Theta_{r+1}
\]
we obtain
\begin{align*}
P_2 d\ce_r & = - (-1)^r P_2 \dT d\Theta_{r+1} \\
& = (-1)^r (\dT P_1 d\Theta_{r+1} - d\Theta_{r+1})
\intertext{so that}
(-1)^{r+1} P_2 d\ce_r & = d\Theta_{r+1} - \dT P_1 d\Theta_{r+1}
\end{align*}
using the homotopy formula of Theorem~\ref{T22}; but
\begin{align*}
\ce_{r+1} & = d\Theta_{r+1} - (-1)^{r+1} \dT\Theta_{r+2} \\
& = d\Theta_{r+1} - \dT P_1 d\Theta_{r+1}
\end{align*}
so that
\[
\ce_{r+1} = (-1)^{r+1} P_2 d\ce_r \, .
\]
Similarly,
\begin{align*}
P_2 d\ce_{m-1} & = - (-1)^{m-1} P_2 \dT d\Theta_m \\
& = (-1)^m d\Theta_m \\
& = (-1)^m \ce_m \, .
\end{align*}
It follows that if $\ce_0 = 0$ then $\ce_m = 0$.
\end{proof}


\bigskip\bigskip\bigskip
\noindent
Department of Mathematics, Faculty of Science\\
The University of Ostrava\\
30.\ dubna 22\\
701 03 Ostrava\\
Czech Republic

\bigskip
\noindent
Email: \url{david@symplectic.demon.co.uk}


\begin{thebibliography}{99}

\bibitem{CS09}
\art{M.\,Crampin, D.J.\,Saunders}{Some concepts of regularity for parametric multiple-integral problems in the calculus of variations}{Czech Math.\,J.}{59 (3)}{2009}{741--758}

\bibitem{GH96}
\book{M.\,Giaquinta, S.\,Hildebrandt}{Calculus of Variations II}{Springer}{1996}

\bibitem{KMS93}
\book{I.\,Kol\'{a}\v{r}, P.W.\,Michor, J.\,Slov\'{a}k}{Natural Operations in Differential Geometry}{Springer}{1993}

\bibitem{Run73}
\book{H.\,Rund}{The Hamilton-Jacobi Equation in the Calculus of Variations}{Krieger}{1973}

\bibitem{Sau09}
\art{D.J.\,Saunders}{Homogeneous variational complexes and bicomplexes}{J.\,Geom.\,Phys.}{59}{2009}{727--739}

\bibitem{Sau10}
\art{D.J.\,Saunders}{Some geometric aspects of the calculus of variations in several independent variables}{Comm.\,Math.}{18 (1)}{2010}{3--19}

\end{thebibliography}
\end{document}